\documentclass[12pt,a4paper,]{article}
\usepackage[latin1]{inputenc}
\usepackage[T1]{fontenc}
\usepackage[english]{babel}
\usepackage{amsmath,amssymb,amsfonts,amsthm}

\usepackage[pdftex]{graphicx}
\usepackage[pdftex]{color}
\usepackage{enumerate}
\usepackage{dsfont}
\usepackage[all]{xy}
\usepackage{anysize}

\marginsize{2.5cm}{2.4cm}{1.6cm}{2.3cm}

\newtheorem{teo}{Theorem}
\newtheorem{deff}[teo]{Definition}
\newtheorem{prop}[teo]{Proposition}

\newtheorem{lemma}[teo]{Lemma}
\newtheorem{cor}[teo]{Corollary}

\newcommand{\Kappa}{\mathcal{K}}
\newcommand{\Deelta}{\mathcal{D}}
\newcommand{\calt}{{\mathcal T}}
\newcommand{\ce}{ \mathcal{E}}
\newcommand{\cf}{ \mathcal{F}}
\newcommand{\cep}{{ \mathcal{E}'}}
\newcommand{\epsil}{\varepsilon}
\newcommand{\bs}{ \backslash }
\newcommand{\ceo}{{ \mathcal{E}_0 }}
\newcommand{\ceop}{ {\mathcal{E}_0' }}
\newcommand{\tx}{{\widetilde{X}}}
\newcommand{\hx}{{\widehat{X}}}
\newcommand{\xid}{{{(X,\widehat {X},\widetilde{X})}}}

\newcounter{contfig}
\newenvironment{figurapng}[3]{\refstepcounter{contfig}\par\vspace{8pt}
\centerline{\includegraphics[scale=#1]{#2}} \centerline{Figure
\thecontfig. #3}}{\vspace{8pt}}

\newenvironment{rem}{\refstepcounter{teo}\noindent\textbf{Remark \theteo.}\par

\vspace{-4pt}\begin{itemize}
\setlength{\itemsep}{-2pt}\item}{\end{itemize}}

\newenvironment{rems}{\refstepcounter{teo}\noindent\textbf{Remarks \theteo.}

\par\vspace{-4pt}\begin{itemize}
\setlength{\itemsep}{-2pt}\item}{\end{itemize}}

\newcommand{\Nset}{\mathbb{N}}

\newcommand{\Rset}{\mathbb{R}}

\DeclareMathOperator{\asdim}{asdim}
\newcommand{\dd}{{\ddabc}}
\DeclareMathOperator{\ddabc}{d} \DeclareMathOperator{\BB}{B}
\DeclareMathOperator{\BBC}{{\overline B}}
\DeclareMathOperator{\mult}{mult} \DeclareMathOperator{\diam}{diam}
\DeclareMathOperator{\mesh}{mesh}

\title{Dugundji's Canonical Covers, asymptotic and covering dimension}
\author{Jesús P. Moreno-Damas}
\date{}

\begin{document}

\maketitle

\parbox{\textwidth}
{\abstract{Given a nowheredense closed subset $X$ of a metrizable
compact space $\tx$, we characterize the dimension of $X$ in terms
of the multiplicity of the canonicals covers of the complementary of
$X$, specially in some particular cases, like when $\tx$ is the
Hilbert cube or the finite dimensional cube and $X$, a Z-set of
$\tx$. In this process, we solve some questions in the literature.}}

\section{Introduction}

This paper reinforces the relation between the canonical covers,
used by Dugundji in \cite{dug2} to prove his famous Extension
Theorem, and the way to describe the dimension of certain subsets of
compact spaces in terms of the corresponding complement. Such a
relation was treated in \cite{cc} in the special case when a compact
metrizable space is embedded as a Z-set in the Hilbert cube.
Concretely, in \cite{cc}, the authors proved that a compact Z-subset
of the Hilbert cube has finite topological dimension if and only if
there exists a canonical cover of finite order in its complement
(for us, the order of a cover $\alpha$ is $1$ less the supremum of
the number elements of $\alpha$ with non empty intersection, that
is, $1$ less the multiplicity of $\alpha$.)

The above kind of relations are in the core of so called Higson-Roe
functor with allows us to compare coarse properties of a coarse
structure with topological properties of the corresponding
Higson-Roe corona (see \cite{cg} and \cite{zs} for a specific
example). In particular,
asymptotic dimensional properties with covering dimensional properties. Several works
go in this way, for example \cite{dr1}, \cite{dr3}, \cite{dr2} and \cite{dr4}.

When the coarse structure is the topological one attached to a
metrizable compactification, Grave proved in \cite{grv} and
\cite{grv2} that the asymptotic dimension of that coarse structure
exceeds just by $1$ the covering dimension of the corona. That
result allowed the authors in \cite{zs} to prove that the gap
between the $C_0$-asymptotic dimension of the complement of a Z-set
in the Hilbert and the topological dimension of the corresponding
Z-set is precisely $1$.

As proved in \cite{cc}, the particularity of the canonical covers in
the complementary is that every one has information about the
dimension of the Z-set. In fact, what was really proved in that
article is that, if a compact Z-set in the Hilbert cube has
dimension $n$, then the order of any canonical cover of the
complement is greater than or equal to $n$ and there exists a
canonical cover which order is less than or equal to $2n+1$.

As pointed out in that article, the definition of canonical cover is in the core of the $C_0$
coarse geometry (see \cite{cg}, \cite{wright} and \cite{wright3} for definitions).
It is natural to ask if, for every canonical cover, the gap between its order and the dimension of the Z-set is at least $1$ and, as it was did in \cite{cc} page
3713, if one can always find a canonical cover
whose gap is just $1$.

In this paper, we use canonical covers to strength and improve
relations pointed out in \cite{cc}. In particular, we are able to
work in a more general framework and to answer positively the
questions mentioned above. We also describe accurately the relations
between the canonical covers and the $C_0$ coarse structure. Finally,
we use our results here to give an easier proof of a Grave's result
in \cite{grv} and \cite{grv2}, relating the topological dimension of the corona with the
corresponding asymptotic dimension.

In section 2 we introduce the basic needed definitions and notation
and state an easy criterium to detect coarse equivalences.

Section 3, which is the core of the paper, contains the main
technical tools and results in this paper. Firstly, we relate the
canonical covers with the coarse structures, getting that a
canonical cover is precisely an open and locally finite cover which
is uniform for the topological coarse structure attached to a
compactification. In our case, that compactification is metrizable
and that means that the cover is uniform for the $C_0$ coarse
structure. That states the relation pointed out in \cite{cc}.

Taking it into account, we are able to relate the multiplicity of
canonical covers with the asymptotic dimension of the $C_0$ coarse
spaces, using the Grave's result mentioned above, to obtain:

\begin{enumerate}[(A) ]
\item If $X$ is a nowheredense closed subset of a metrizable compact
space $\tx$,  then $\dim X\leq n$ if and only if for every canonical
cover $\alpha$ of $\tx\bs X$, there exists a canonical cover $\beta$
which has $\alpha$ as a refinement and whose order is less
than or equal to $n+1$.\end{enumerate}

To improve this equivalence when $\tx$
is the Hilbert cube and $X$ is a Z-set of $\tx$, we define a more general category, the
\emph{cylindrical subsets}. This allows us to extends results which work
for the Z-sets of the Hilbert cube, to other cases, like the Z-sets
of the finite dimensional cube $[0,1]^n$. We get:

\begin{enumerate}[(B) ]
\item If $X$ is a cylindrical, nowheredense and closed subset of a metrizable
compact space $\tx$ (in particular, if $X$ is a Z-set of $\tx$ and $\tx$ is the Hilbert Cube
or the finite dimensional cube), then $\dim X\leq n$
if and only if there exists a canonical cover of $\tx\bs X$ whose
order is less than or equal to $n+1$.\end{enumerate}

Actually, (B) is satisfied not only for canonical covers, but also for
open and $C_0$-uniform covers. Consequently, the order of every open and uniform cover
is at least the $C_0$-asymptotic dimension. That means that, when $X$ is a cylindrical subset,
from the asymptotic dimensional point of view, the open and uniform covers
are big enough. That sugests a natural question: In the general case,
when is a uniform cover big enough from the asymptotic dimensional
point of view? We finish the section by answering that.

In section 4 we recover Grave's result in an easier way. To do it,
we give some results in Topological Dimension Theory
(one of them, developed by us in \cite{moreno}).

\section{Preliminaries: Basic definitions and notations}

We say that \textit{compactification pack}  is a vector $\xid$ such
that $\tx$ is a compact Hausdorff space, $X$ is a nowheredense
closed subset of $\tx$ and $\hx=\tx\bs X$. Observe that $\tx$ is a
compactification of $\hx$ and $X$ is its corona.

For us, a cover $\alpha$ of a set $Z$ is a collection of subsets of
$Z$ whose union is $Z$.

If $Z$ is a topological space and $\alpha$ is a family of subsets of
$Z$, we say that $\alpha$ is open if every $U\in\alpha$ is open. If
$Z$ is a metric space, we say that $\mesh\alpha=\sup\{\diam
U:U\in\alpha\}$.

If and $A\subset Z$ and $\alpha$ and $\beta$ are families of subsets
of $Z$, we denote $\alpha(A)=\bigcup_{\substack{U\in\alpha\\U\cap
A\neq\varnothing}}U$. By $\beta\prec\alpha$ we mean that $\beta$ is
a refinement of $\alpha$. The multiplicity of $\alpha$ is
$\mult\alpha=\sup \{\#A:A\subset\alpha,\bigcap_{U\in
A}U\neq\varnothing\}$ (where, for every set $B$, $\# B$ means the
cardinal of $B$).

The concept of \textit{canonical cover} was used by Dugundji in
\cite{dug2}. In our language, a canonical cover $\alpha$ of a
compactification pack $\xid$ is an open and locally finite cover of
$\hx$ such that for every $x\in X$ and every neighborhood $V_x$ of
$x$ in $\tx$ there exists a neighborhood $W_x$ of $x$ in $\tx$ with
$\alpha(W_x)\subset V_x$.

Let us give some definitions of coarse geometry. For more
information, see \cite{cg}. Let $E,F\subset Z\times Z$, let $x\in Z$
and let $K\subset Z$. The product of $E$ and $F$, denoted by $E\circ
F$, is the set $\{(x,z):\exists y\in Z\textrm{ such that }(x,y)\in
E,(y,z)\in F\}$, the inverse of $E$, denoted by $E^{-1}$, is the set
$E^{-1}=\{(y,x):(x,y)\in E\}$, the diagonal, denoted by $\Delta$, is
the set $\{(z,z):z\in Z\}$. If $x\in Z$, the $E$-ball of $x$,
denoted by $E_x$ is the set $E_x=\{y:(y,x)\in E\}$ and, if $K\subset
Z$, $E(K)$ is the set $\{y:\exists x\in K\textrm{ such that }
(y,x)\in E\}$. We say that $E$ is symmetric if $E=E^{-1}$.

A coarse structure $\ce$ over a set $Z$ is a family of subsets of
$Z\times Z$ which contains the diagonal and is closed under the
formation of products, finite unions, inverses and subsets. The
elements of $\ce$ are called controlled sets. $B\subset Z$ is said
to be bounded if there exists $x\in Z$ and $E\in\ce$ with $B=E_x$
(equivalently, $B$ is bounded if $B\times B\in\ce$).

A map $f:(Z,\ce)\rightarrow (Z',\cep)$ between coarse spaces is
called coarse if $f\times f(E)$ is controlled for every controlled
set $E$ of $Z$ and $f^{-1}(B)$ is bounded for every bounded subset
$B$ of $Z'$. We say that $f$ is a coarse equivalence if $f$ is
coarse and there exists a coarse map $g:(Z',\cep)\rightarrow
(Z,\ce)$ such that $\{(g\circ f(x),x):x\in Z\}\in\ce$ and $\{(f\circ
g(y),y):y\in Z'\}\in \cep$. In this case, $g$ is called a coarse
inverse of $f$.

A subset $A\subset Z$ is coarse dense if $Z=E(A)$ for some $E\in\ce$ .

Given $E\subset Z\times Z$, we denote $\Kappa(E)=\{E_x:x\in Z\}$. If
$\alpha$ is a family of subsets of $Z$, we say that
$\Deelta(\alpha)=\bigcup_{U\in\alpha}U\times U$ (observe that
$(x,y)\in\Deelta(\alpha)$ if and only if $x,y\in U$ for some
$U\in\alpha$).

A family of subsets $\alpha$ of $Z$ is called uniform if
$\Deelta(\alpha)\in\ce$. Note that, if $E\in\ce$, then $\Kappa(E)$
is uniform. Dydak and Hoffland showed in \cite{lss} that the coarse
structures can be described in terms of the uniform covers.

Intuitively, the uniform families of subsets of $(Z,\ce)$ behave
like the controlled sets. For example, suppose that
$f:(Z,\ce)\rightarrow (Z',\ce')$ is a coarse map between are coarse
spaces, $E$ is a controlled set of $(Z,\ce)$, $\alpha$ is a uniform
family of subsets of $(Z,\ce)$ and $\beta$ is a family of subsets of
$Z$ such that $\beta\prec\alpha$. Then, $E(\alpha)$ and $\beta$ are
uniform families of subsets of $(Z,\ce)$ and
$f(\alpha)=\{f(U):U\in\alpha\}$ is a uniform family of subsets of
$(Z',\ce')$.  For every $B\subset Z$, $B$ is bounded if and only if
there exists a uniform family of subsets $\gamma$ of $(Z,\ce)$ such
that $B\in\gamma$ (equivalently, if $\{B\}$ is uniform).

If $Z$ is a topological space and $E\subset Z\times Z$, we say that
$E$ is proper if $E(K)$ and $E^{-1}(K)$ are relatively compact for
every relatively compact subset $K\subset Z$. If $\ce$ is a coarse
structure over $Z$, we say that $(Z,\ce)$ is a proper coarse space
if $Z$ is Hausdorff, locally compact and paracompact, $\ce$ contains
a neighborhood of the diagonal in $Z\times Z$ and the bounded
subsets of $(Z,\ce)$ are precisely the relatively compact subsets of
$Z$.

By $Q$ we denote the Hilbert cube $[0,1]^\Nset$. Let $\tx=[0,1]^n$
with $n\in\Nset$ or $\tx=Q$ and suppose that $\dd$ is a metric in
$\tx$. $X$ is a Z-set of $\tx$ if it is closed and for every
$\epsil>0$ there exists a continuous function $f:\tx\rightarrow\tx$
such that $\dd'(f,Id)<\epsil$
---where $\dd'$ is the supreme metric--- and $f(\tx)\cap
X=\varnothing$ (the definition of Z-set given in \cite{hc}, chapter
I-3, pag. 2, is equivalent in this context).

For us, an increasing function between two ordered sets
$f:(X,<)\rightarrow (X',<')$ is a function such that $x<y$ implies
$f(x)\leq' f(y)$.

The following identities will be useful along this article. Let
$Z,Z'$ be sets, suppose $x\in Z$, $x'\in Z'$, $A,B\subset Z$ and
 $E,F\subset Z\times Z$. Consider a family of subsets $\alpha$ of $Z$ and a map $f:Z\rightarrow Z'$. Then:

\begin{equation}\label{usefulidentities1}
\Deelta(\alpha)(A)=\alpha(A)
\end{equation}
\begin{equation}\label{usefulidentities2}
(E\circ F)_x=E(F_x)
\end{equation}
\begin{equation}\label{usefulidentities6}
E(A)=\bigcup_{a\in A}E_a
\end{equation}
\begin{equation}\label{usefulidentities3}
E(A)\cap B\neq\varnothing \textrm{ if and only if }
A\cap E^{-1}(B)\neq\varnothing.
\end{equation}
\begin{equation}\label{usefulidentities4}
E(B)\subset A \textrm{ if and only if } E\cap (Z\bs A)\times
B=\varnothing
\end{equation}
\begin{equation}\label{usefulidentities5}
(f\times f(E))_{x'}=\bigcup_{a\in f^{-1}(x')}f(E_a)=f(E(f^{-1}(x')))
\end{equation}

The following proposition provides a criterion to detect coarse
equivalences:

\begin{prop}\label{criterio} If
$f:(X,\ce)\rightarrow (Y,\cf)$ is a map between coarse spaces, then
\begin{enumerate}[a) ]
\item $f$ is a coarse equivalence.
\item There exists $g:Y\rightarrow X$ such that $f\times
f(E)\in \cf$ for every $E\in\ce$, $g\times g(F)\in\ce$ for every
$F\in\cf$, $\{(x,g\circ f(x)):x\in X\}\in\ce$ and $\{(y,f\circ
g(y)):y\in Y\}\in\cf$.
\item $f\times f(E)\in\cf$ for every $E\in\ce$,
$(f\times f)^{-1}(F)\in\ce$ for every $F\in\cf$ and there exists
$g:Y\rightarrow X$ such that $\{(y,f\circ g(y)):y\in Y\}\in\cf$.
\item $f\times f(E)\in\cf$ for every $E\in\ce$,
$(f\times f)^{-1}(F)\in \ce$ for every $F\in\cf$ and $f(X)$ is
coarse dense in $(Y,\cf)$. \end{enumerate} are equivalent. Moreover,
if $g$ is the map of b) or c), then $g$ is a coarse inverse of
$f$.\end{prop}
\begin{proof} Throughout the proof, every time we use a map called $g:Y\rightarrow X$,
we will denote by $G$ and $H$ the sets
$$G=\{(x,g\circ f(x)):x\in X\}$$
$$H=\{(y,f\circ g(y)):y\in Y\}$$

It is obvious that a) implies b).

To see that b) implies a), it is suffices to show that $f^{-1}(U)$
is bounded for every bounded subset $U$ of $X$ and $g^{-1}(V)$ is
bounded for every bounded subset $V$ of $Y$. But it is easily
deduced from:
$$f^{-1}(V)\times f^{-1}(V)\subset G\circ (g\times g(V\times V))\circ
G^{-1}\in\ce$$
$$g^{-1}(U)\times g^{-1}(U)\subset H\circ (f\times f(U\times U))\circ
H^{-1}\in\cf$$

Moreover, $g$ is a coarse inverse of $f$.

To see that b) implies c) it is sufficient to show that $(f\times
f)^{-1}(F)\in \ce$ for every $F\in \cf$. But it follows from:
$$(f\times f)^{-1}(F)\subset G\circ (g\times g(F))\circ
G^{-1}\in\ce$$

To see that c) implies b), it is sufficient to show that $g\times
g(F)\in\ce$ for every $F\in\cf$ and that $G\in\ce$. But it follows
from:
$$g\times g(F)\subset (f\times f)^{-1}(H^{-1}\circ F\circ H)\in\ce$$
$$G\in (f\times f)^{-1}(H)\in\ce$$

Moreover, since $g$ satisfies b), $g$ is a coarse inverse of $f$.

To see that c) implies d), it suffices to show that $f(X)$ is coarse
dense in $(Y,\cf)$. And it happens, since $Y=H(f(X))$.

To see that d) implies c) we only need to define a map
$g:Y\rightarrow X$ such that $H\in\cf$. To show that, take $M\in\cf$
such that $Y=M(f(X))$. By (\ref{usefulidentities6}),
$Y=\bigcup_{y\in f(X)}M_y$, so there exists a partition of $Y$
 $\{P_y:y\in f(X)\}$ such that
$P_y\subset M_y$ for every $y\in f(X)$. Choose $x_y\in f^{-1}(y)$
for every $y\in f(X)$. Let $g:Y\rightarrow X$ be the map such that,
for every $y\in f(X)$ and every $y'\in P_y$, $g(y')=x_y$.

Observe that $(y',f\circ g(y'))=(y',f(x_y))=(y',y)\in
P_y\times\{y\}\subset M_y\times\{y\}\subset M$. Then, $H\subset
M\in\cf$ and $H\in\cf$. \end{proof}

\begin{cor} If $f:(X,\ce)\rightarrow (Y,\cf)$ is a bijective coarse equivalence, then $f^{-1}$ is a coarse inverse of $f$.\end{cor}
\begin{proof} Since $f$ is a coarse equivalence, it satisfies
property d) of Proposition \ref{criterio}. Then, $f$ and $f^{-1}$
satisfice property c) and hence, $f^{-1}$ is a coarse inverse of
$f$.\end{proof}

\begin{cor}\label{corofequivalpha}If $f:(X,\ce)\rightarrow (Y,\cf)$
is a coarse equivalence and $\alpha$ is a uniform cover of
$(Y,\cf)$, then $f^{-1}(\alpha)$ is a uniform cover of $X$.\end{cor}

\begin{proof} Since $\Deelta(f^{-1}(\alpha))=(f\times
f)^{-1}(\Deelta(\alpha))$ is a controlled set of $(X,\ce)$, it
follows that $f^{-1}(\alpha)$ is uniform.\end{proof}

\section{Canonical Covers, Asymptotic Dimension and Dimension}

\subsection{Canonical Covers and coarse structures}

Firstly, we relate the canonical covers with the $C_0$ coarse
structures.

Recall the following definition, see \cite{cg}:

\begin{deff}\label{defcontcoarest} Let $\xid$ be a compactification pack. The
topological coarse structure $\ce$ over $\hx$ attached to the
compactification $\tx$ is the collection of all $E\subset \hx\times
\hx$ satisfying any of the following equivalent properties:
\begin{enumerate}[a) ]
\item $Cl_{\tx\times\tx}E$ meets $\tx\times\tx\bs\hx\times\hx$ only
in the diagonal of $X\times X$.

\item $E$ is proper and for every net $(x_\lambda,y_\lambda)\subset
E$, if $\{x_\lambda\}$ converges to a point $x$ of $X$, then
$y_\lambda$ converges also to $x$.

\item $E$ is proper and for every point $x\in X$ and every
neighborhood $V_x$ of $x$ in $\tx$ there exists a neighborhood $W_x$
of $x$ in $\tx$ such that $E(W_x)\subset V_x$.\end{enumerate}\end{deff}

\begin{rems}
In this coarse structure, the bounded subsets of $\ce$ are precisely
the relatively compact subsets of $\hx$. Moreover, if $\tx$ is
metrizable, then $\ce$ is proper.

\item The definition above is Proposition 2.27 and Definition 2.28 of \cite{cg}
(pags. 26-27) together with the author's correction in \cite{cg3}.
The property c) above is not the one of that proposition, but it is
equivalent (see (\ref{usefulidentities4})).\end{rems}

\begin{prop}\label{caractce} Let $\xid$ be a compactification pack. Suppose that $\ce$ is the topological coarse
structure over $\hx$ attached to the compactification $\tx$ and let
$E\subset \hx\times \hx$. Then, $E\in\ce$ if and only if, for every
$x\in X$ and every neighborhood $V_x$ of $x$ in $\tx$, there exist a
neighborhood $W_x$ of $x$ in $\tx$ such that $(E\cup
E^{-1})(W_x)\subset V_x$.
\end{prop}

\begin{proof}
If $E\in\ce$, then also $E\cup E^{-1}\in \ce$, thus $E\cup E^{-1}$
satisfies property c) of definition \ref{defcontcoarest} and we get
the necessity.

Let us see the sufficiency. To prove that $E$ satisfies property c)
of definition \ref{defcontcoarest}, it suffices to show that $E$ is
proper, that is, $E(K)\cup E^{-1}(K)$ is relatively compact in $\hx$
for every relatively compact subset $K$ of $\hx$.

Take a relatively compact subset $K$ of $\hx$. Fix $x\in X$. Since
$\tx\bs K$ is a neighborhood of $x$ in $\tx$, there exists an open
neighborhood $W_x$ of $x$ in $\tx$ such that $(E\cup
E^{-1})(W_x)\subset \tx\bs K$, i. e. $K\cap (E\cup
E^{-1})(W_x)=\varnothing$. By (\ref{usefulidentities3}), it is
equivalent to $(E\cup E^{-1})(K)\cap W_x=\varnothing$. Since
$\bigcup_{x\in X}W_x$ is an open neighborhood of $X$ in $\tx$ with
$(E\cup E^{-1})(K)\cap \left(\bigcup_{x\in
X}W_x\right)=\varnothing$, we have that $(E\cup E^{-1})(K)=E(K)\cup
E^{-1}(K)$ is relatively compact in $\hx$.
\end{proof}

\begin{cor}\label{caractcanon} Let $\xid$ be a compactification pack and suppose that $\ce$ is the topological
coarse structure over $\hx$ attached to the compactification $\tx$.
If $\alpha$ is a family of subsets of $\hx$, then $\alpha$ is
uniform if and only if for every $x\in X$ and every neighborhood
$V_x$ of $x$ in $\tx$, there exists a neighborhood $W_x$ of $x$ in
$\tx$ such that $\alpha(W_x)\subset V_x$.

Particulary, $\alpha$ is a canonical cover of $\xid$ if and only if
it is an open and locally finite cover of $\hx$ which is controlled
for $\ce$.
\end{cor}

\begin{proof} $\alpha$ is uniform if and only if $\Deelta(\alpha)\in
\ce$. By Proposition \ref{caractce}, (\ref{usefulidentities1}) and
the symmetry of $\Deelta(\alpha)$, that is equivalent to say that
for every $x\in X$ and every neighborhood of $x$ in $\tx$ there
exists a neighborhood $W_x$ of $x\in\tx$ such that
$\Deelta(\alpha)(W_x)\subset V_x$, i. e. $\alpha(W_x)\subset V_x$.
\end{proof}

Recall the following definition of Wright in \cite{wright} or
\cite{wright3} (see also of example 2.6 of \cite{cg}, pag. 22):

\begin{deff} Let $(\hx,\dd)$ be a metric space. The $C_0$ coarse
structure is the collection of all subsets $E\subset\hx\times\hx$
such that for every $\epsil>0$ there exists a compact subset $K$ of
$\hx$ such that $\dd(x,y)<\epsil$ whenever $(x,y)\in E\bs K\times
K$.\end{deff}

We have the following using the same argument as in \cite{zs}:

\begin{prop}
Let $\xid$ be a metrizable compactification pack and let $\dd$ be a
metric on $\tx$. Then, the topological coarse structure over $\hx$
attached to the compactification $\tx$ is the $C_0$ coarse structure
over $\hx$ attached to $\dd$.\end{prop}

\begin{proof} The proof of Proposition 6 of \cite{zs} is valid in this context.
But, using Proposition \ref{caractce}, we get the following shorter
proof:

Denote by $\ce$ and $\ceo$ the topological coarse structure attached
to $\tx$ and the $C_0$ coarse structure attached to $\dd$
respectively.

Let $E\in\ceo$. Observe that $E\cup E^{-1}\in\ceo$. Fix $x\in X$ and
a neighborhood $V_x$ of $x$ in $\tx$. Choose an $\epsil>0$ such that
$\BB(x,\epsil)\subset V_x$ and a compact subset $K$ of $\hx$ such
that $\dd(x,y)<\frac{\epsil}{2}$ whenever $(x,y)\in (E\cup
E^{-1})\bs K\times K$. Take
$\delta=\min\left\{\frac{\epsil}{2},\frac{\dd(K,X)}{2}\right\}$ and
$W_x=\BB(x,\delta)$. Pick a point $y\in (E\cup E^{-1})(W_x)$ and
take $z\in W_x$ such that $(y,z)\in E\cup E^{-1}$. Since $z\not\in
K$, it follows that $y,z\in (E\cup E^{-1})\bs K\times K$ and
$\dd(y,z)<\frac{\epsil}{2}$. Hence
$\dd(y,x)\leq\dd(y,z)+\dd(z,x)<\frac{\epsil}{2}+\frac{\epsil}{2}=\epsil$
and, consequently, $(E\cup E^{-1})(W_x)\subset \BB(x,\epsil)\subset
V_x$. Therefore, $E\in\ce$.

Take now $E\in\ce$ and fix $\epsil>0$. For every $x\in X$, consider
a neighborhood $W_x$ of $x$ contained in
$\BB\left(x,\frac{\epsil}{2}\right)$ such that $(E\cup
E^{-1})(W_x)\subset \BB\left(x,\frac{\epsil}{2}\right)$. Let
$K=\tx\bs\bigcup_{x\in X}W_x$ and suppose $(y,z)\in E\bs K\times K$.
Observe that neither $y\in K$ nor $z\in K$. If $z\not\in K$, then
there exists $x\in X$ with $z\in W_x$ and $y\in E(W_x)$. If
$y\not\in K$, then there exists $x\in X$ such that $y\in W_x$ and
$z\in E^{-1}(W_x)$. In both cases, $y,z\in W_x\cup (E\cup
E^{-1})(W_x)\subset \BB\left(x,\frac{\epsil}{2}\right)$ and hence
$\dd(y,z)\leq\dd(y,x)+\dd(x,z)\leq\frac{\epsil}{2}+\frac{\epsil}{2}=\epsil$.
Therefore, $E\in\ceo$.
\end{proof}

From now on, we will use the following notation:

\begin{deff} Let $\xid$ be a metrizable compactification pack. The $C_0$ coarse structure over
$\hx$ attached to $\xid$, denoted by $\ceo\xid$ or by $\ceo$ when no
confusion can arise, is the topological coarse structure attached to
the compactification $\tx$, i. e. the $C_0$ coarse structure
attached to any metric of $\tx$ restricted to $\hx$.\end{deff}

We use this notation because working with a metric in $\tx$
simplifies the calculus and is easier to see geometrically (see
figure \ref{ejgraficocanonico}).

\begin{figurapng}{.85}{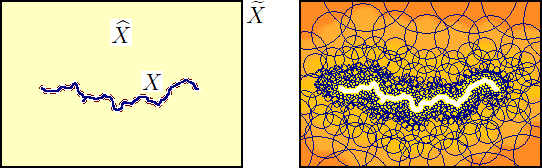}{$\xid$ and uniform cover of $\ceo\xid$} \label{ejgraficocanonico}\end{figurapng}\smallskip

\begin{prop}\label{recunif}
Let $(\hx,\dd)$ be a locally compact metric space and let $\ceo$ be
the $C_0$ coarse structure. Consider a family of subsets $\alpha$ of
$\hx$. Then, $\alpha$ is uniform for $\ceo$ if and only
if:\begin{itemize}
\item $\alpha$ is proper, i. e. for every relatively compact subset $K$ of $\hx$ ,
$\alpha(K)$ is relatively compact in $\hx$.
\end{itemize}
and
\begin{itemize}
\item For every $\epsil>0$ there exists a compact subset $K$ of $\hx$ such that
$\diam U<\epsil$ for every $U\in\alpha$ with $U\cap
K=\varnothing$.\end{itemize}\end{prop}

\begin{proof}
Suppose that $\alpha$ is uniform. Since $\Deelta(\alpha)$ is
controlled and hence proper, (\ref{usefulidentities1}) shows that
$\alpha(K)=\Deelta(\alpha)(K)$ is relatively compact for every
relatively compact  $K$ of $\hx$.

Let $\epsil>0$ and consider a compact subset $K$ of $\hx$ such that
$\dd(x,y)<\epsil$ for every $(x,y)\in\Deelta(\alpha)\bs K\times K$.
Take $U\in \alpha$ with $U\cap K=\varnothing$ and $x,y\in U$. Since
$x,y\not\in K$, $(x,y)\in \Deelta(\alpha)\bs K\times K$ and hence
$\dd(x,y)<\epsil$. Thus, $\diam U\leq \epsil$.

For the reciprocal, fix $\epsil>0$ and consider a compact subset $K$
of $\hx$ with $\diam U<\epsil$ for every $U\in\alpha$ such that
 $U\cap K=\varnothing$.
Since $\alpha$ is proper, $\alpha(K)$ is relatively compact.
Consider $K'=\overline{\alpha(K)}$ and pick
$(x,y)\in\Deelta(\alpha)\bs K'\times K'$. Observe that neither $x\in
K$ nor $y\in K$. Suppose, without loss of generality, that $x\not\in
K'$. Let $U\in\alpha$ be such that $x,y\in U$. Since $x\in U$
and $x\not\in \alpha(K)=\bigcup_{\substack{U\cap K\neq\varnothing\\
U\in \alpha}} U$, we have that $U\cap K=\varnothing$. Then,
$\dd(x,y)\leq\diam U<\epsil$. Hence,
$\Deelta(\alpha)\in\ceo$.\end{proof}

\begin{lemma}\label{basicodefrec} Let $\xid$ be a metrizable compactification
pack and consider $(\hx,\ceo)$.

Fix a sequence of open subsets $\{W_n\}_{n=0}^\infty$ of $\tx$ with
$W_0=\tx$, $W_0\supset \overline W_1\supset W_1\supset \overline
W_2\supset W_2\supset\dots$ and $X=\bigcap_{n=0}^\infty W_n$.

Let $\{\beta_n\}_{n=0}^\infty$ be a sequence of families of open
subsets of $\tx$ with $\beta_0=\{\tx\}$, $X\subset
\bigcup_{U\in\beta_n}U$ for every $n$ and
$\lim_{m,n\rightarrow\infty}\mesh\{U\cap W_m:U\in\beta_n\}=0$.

Consider $\alpha(\{\beta_n\},\{W_n\})=\{U\cap (W_n\bs \overline
W_{n+2}):U\in\beta_n, n\geq 0\}$. Then:
\begin{enumerate}[a) ]
\item $\alpha(\{\beta_n\},\{W_n\})$ is an open and uniform family of subsets of
$\hx$.

\item If each $\beta_n$ if finite, $\alpha(\{\beta_n\},\{W_n\})$ is
locally finite.

\item If $\gamma$ is a uniform family of subsets of $\hx$,
then there exists a subsequence $\{n_k\}_{k=0}^\infty$ with $n_0=0$
such that $\gamma\prec
\alpha(\{\beta_k\},\{W_{n_k}\})$.\end{enumerate}\end{lemma}

\begin{figurapng}{.75}{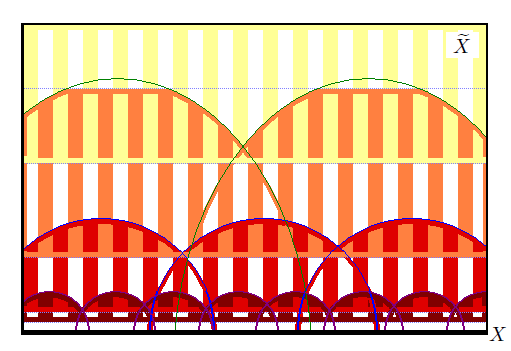}{Cover
$\alpha\big(\left\{\beta_n\right\},\left\{W_n\right\}\big)$}\label{figuraescamas}\end{figurapng}
\begin{proof}

Observe that if  $K$ is a compact subset of $\hx$, then $\{K\cap
\overline W_n\}_{n=0}^\infty$ is a sequence of nested compact
subsets such that $\bigcap_{n=0}^\infty K\cap \overline
W_n=\varnothing$ and, hence there exists $N\in\Nset$ with
$K\cap\overline W_N=\varnothing$. Equivalently, if $U$ is an open
subset of $\tx$ containing $X$, then there exist $N\in\Nset$ such
that $W_N\subset U$.

For short, let us denote $\alpha(\{\beta_n\},\{W_n\})$ by $\alpha$. Obviously,
$\alpha$ is open.

Let us see that $\alpha$ is uniform, i. e. $\Deelta(\alpha)\in\ceo$.
Fix $\epsil>0$ and let $N\in\Nset$ be such that for every $m,n\geq N$,
$\mesh\{U\cap W_m:U\in\beta_n\}<\epsil$. Put $K=\tx\bs W_{N+2}$,
pick $(x,y)\in\Deelta(\alpha)\bs K\times K$ and take $V\in\alpha$
with $x,y\in V$. Let $n>0$ and $U\in\beta_n$ be such that
 $V=U\cap W_n\bs \overline W_{n+2}$. Since $(x,y)\not\in K\times
K$, neither $x\in K$ nor $y\in K$. Suppose, without loss of
generality, that $x\not\in K$, or equivalently, that $x\in W_{N+2}$.
Since $x\in V\subset\tx\bs W_{n+2}$, we have that $W_{N+2}\bs
W_{n+2}\supset\{x\}\neq\varnothing$ and, consequently, $n> N$. Thus,
$\dd(x,y)\leq\diam U\cap W_n<\epsil$ and $\Deelta(\alpha)\in\ceo$.

Fix $x\in \hx$ and take $n\geq 0$ such that $x\in W_n\bs W_{n+1}$.
Consider the neighborhood of $x$ $W_n\bs \overline W_{n+2}$. It is
easy to check that $\{V\in\alpha: V\cap W_n\bs \overline
W_{n+2}\neq\varnothing\}\subset \{U\cap (W_k\bs \overline
W_{k+2}):U\in\beta_k, k=n-1,n,n+1\}$. Then, if each $\beta_n$ is
finite, $\alpha$ is locally finite.

Let $\gamma$ be a uniform family of subsets of $(\hx,\ceo)$. Let us
construct a subsequence $\{n_k\}_{k=0}^\infty$ with $n_0=0$ such
that $\gamma\prec\alpha\big(\{\beta_k\},\{W_{n_k}\}\big)$. We will
use the characterization of uniform families of subsets given in
Lemma \ref{recunif}.

Let $L_n=\sup\{\diam V:V\in \gamma, V\subset W_n\}$ for every
$n\in\Nset\cup\{0\}$. Let us see that $L_n\rightarrow 0$. Fix
$\epsil>0$ and suppose that $K$ is a compact subset of $\hx$ such
that $\diam V<\epsil$ for every $V\in\alpha$ with $V\cap
K=\varnothing$. Choose $N\in\Nset$ with $K\cap W_N=\varnothing$. Fix
$n\geq N$ and $V\in\gamma$ with $V\subset W_n$. Since $V\cap
K=\varnothing$, we have that $\diam V<\epsil$. Consequently,
$L_n\leq\epsil$ for every $n\geq N$.

Put $n_0=0$. Assuming that $n_0,\cdots,n_{k-1}$ are defined, let us
define $n_k$:

Choose $m>0$ such that $\overline W_m\subset\bigcup_{U\in\beta_k}U$.
Consider the cover of $\tx$
 $\beta_k\cup \{\tx\bs\overline
W_m\}$ and let $L$ be its Lebesgue number. Take $m'$ such that
$L_n<L$ for every $n\geq m'$.

Since $\gamma$ is a uniform cover of $\ceo$, it is proper. Since
$\tx\bs W_{n_{k-1}}$ is relatively compact in $\hx$, it follows that
$\gamma(\tx\bs W_{n_{k-1}})$ is relatively compact in $\hx$. Choose
$m''$ with $\overline W_{m''}\cap \gamma(\tx\bs
W_{n_{k-1}})=\varnothing$.

Let $n_k=\max\{n_{k-1}+1,m+1,m',m''\}$. By definition of $m''$:
\begin{equation}\label{enfadf42qee1}
\gamma(\tx\bs W_{n_{k-1}})\subset \tx\bs\overline W_{n_k}
\end{equation}

Let us see that:
\begin{equation}\label{enfadf42qee2}
\{V\in\gamma:V\subset W_{n_k}\}\prec\beta_k
\end{equation}

Fix $V\in\gamma$ with $V\subset W_{n_k}$. Since $n_k\geq m'$, $\diam
V\leq L_{n_k}< L$. Then, there exists
$U\in\beta_k\cup\{\tx\bs\overline W_m\}$ such that $V\subset U$.
Necessarily $U\in \beta_k$, because $V\cap (X\bs\overline
W_m)=\varnothing$.

Let us see now that
$\gamma\prec\alpha\big(\{\beta_k\},\{W_{n_k}\}\big)$. Fix
$V\in\gamma$. Since $V$ is bounded and hence relatively compact,
there exists $N$ with $W_n\cap V=\varnothing$. Then, $V\cap (\tx\bs
W_n)\neq\varnothing$ for every $n\geq N$. Let $k=\min\{k':V\cap
(\tx\bs W_{n_{k'}})\neq\varnothing\}$. By (\ref{enfadf42qee1}),
\begin{equation}\label{enfadf42qee3}
V\subset\gamma(\tx\bs W_{n_k})\subset  \tx\bs \overline W_{n_{k+1}}
\end{equation}

If $k=0$, then (\ref{enfadf42qee3}) implies that $V\subset \tx\bs
\overline W_{n_1}\subset \tx\cap (W_{n_0}\bs \overline
W_{n_2})\in\alpha\big(\{\beta_k\},\{W_{n_k}\}\big)$.

Suppose $k\geq 1$. By definition of $k$, $V\cap (\tx\bs
W_{n_{k-1}})=\varnothing$, that is, $V\subset W_{n_{k-1}}$. By
(\ref{enfadf42qee3}), $V\subset W_{n_{k-1}}\bs \overline
W_{n_{k+1}}$.

Since $V\subset W_{n_{k-1}}$, by applying (\ref{enfadf42qee2}) to
$k-1$, we get that there exists $U\in \beta_{k-1}$ such that
$V\subset U$. Therefore, $V\subset U\cap (W_{n_{k-1}}\bs\overline
W_{n_{k+1}})\in\alpha\big(\{\beta_k\},\{W_{n_k}\}\big)$.
\end{proof}

\begin{prop}\label{muchoscanonicos} Let $\xid$ be a metrizable compactification pack
and consider $(\hx,\ceo)$. For every uniform family $\gamma$ of
subsets of $\hx$, there exists a canonical cover $\alpha$ with
$\gamma\prec\alpha$.\end{prop}

\begin{proof}Set $W_0=\tx$ and $\beta_0=\{\tx\}$. Let $k=\sup_{x\in
\tx}\dd(x,X)$. For every $n\in\Nset$, put
$W_n=\BB\left(X,\frac{k}{2n}\right)$. By the compactness of $X$, for
every $n\in\Nset$ we may choose be a finite subfamily $\beta_n$  of
$\{\BB(x,\frac{1}{n}):x\in X\}$ such that $X\subset
\bigcup_{U\in\beta_n}U$.

By Lemma \ref{basicodefrec}, there exists a subsequence $\{n_j\}$
with $n_0=0$ such that $\gamma\cup\{\{x\}:x\in
\hx\}\prec\alpha\big(\{\beta_j\},\{W_{n_j}\}\big)$. Moreover,
$\alpha$ is open, locally finite and uniform for $\ceo$.

Since $\gamma\cup\{\{x\}:x\in X\}$ is a cover of $\hx$, $\alpha$ is
a cover too. By Corollary \ref{caractcanon},
$\alpha\big(\{\beta_j\},\{W_{n_j}\}\big)$ is a canonical cover.
\end{proof}

Now, given a metrizable compactification pack $\xid$, we relate the
dimension of $X$ with the order of the canonical covers in $\hx$.

\begin{deff}
Let $Z$ be a set, suppose $x\in Z$, $V\subset Z$, $E\subset Z\times
Z$ and that $\alpha$ is a family of subsets of $Z$. Then:
 \begin{itemize}
 \item $\mult_x\alpha=\#\{U\in\alpha:x\in V\}$
 \item $\mult_{V}\alpha=\#\{U\in\alpha:V\cap U\neq\varnothing\}$
 \item $\mult_{E}\alpha=\sup\{\mult_{E_x}\alpha:x\in Z\}$
\end{itemize}
\end{deff}

\begin{deff} Let $(Z,\ce)$ be a coarse space. We say that $\asdim
(Z,\ce)\leq n$ if $(Z,\ce)$ satisfies any of the following
equivalent properties:
\begin{enumerate}[a) ]
\item For every uniform cover $\beta$ there exists a uniform cover
$\alpha\succ\beta$ with $\mult\alpha\leq n+1$.
\item For every controlled set $E$
there exists a uniform cover $\alpha$ such that $\mult_E\alpha\leq
n+1$.
\end{enumerate}
\end{deff}

The equivalence of the properties above is given in \cite{grv} or
\cite{grv2} (Chapter 3.2 or \cite{grv}, pags 35-39).

From \cite{grv} or \cite{grv2} we take the following theorem
(Theorem 5.5 or \cite{grv}, pag. 59), adapted to our language:

\begin{teo}[Grave's theorem] If $\xid$ be a metrizable compactification pack, then
$$\asdim (\hx,\ceo)=\dim X+1$$
\end{teo}

\begin{lemma}\label{cambiomult}If $\alpha$ is a family of subsets
of a set $Z$ and $E,F\subset Z$, then:
$$\mult_F E(\alpha)\leq\mult_{E^{-1}\circ F}\alpha$$
$$\mult E(\alpha)\leq\mult_{E^{-1}}\alpha$$
\end{lemma}
\begin{proof}
Let $U\subset Z$. From (\ref{usefulidentities2}) and
(\ref{usefulidentities3}) we get the equivalences: $U\cap
(E^{-1}\circ F)_x\neq\varnothing$ $\Leftrightarrow$ $U\cap
E^{-1}(F_x)\neq\varnothing$ $\Leftrightarrow$ $E(U)\cap
F_x\neq\varnothing$.

Then, $\mult_{(E^{-1}\circ F)_x}\alpha=\#\{U\in\alpha:U\cap
(E^{-1}\circ F)_x\neq\varnothing\}=\#\{U\in\alpha: E(U)\cap
F_x\neq\varnothing\}\geq \#\{V\in E(\alpha):V\cap
F_x\neq\varnothing\}=\mult_{F_x} E(\alpha)$. Taking supreme over $x$
we get $\mult_{E^{-1}\circ F}\alpha\geq\mult_F E(\alpha)$.

The second inequality is deduced from the first one taking into
account that $\mult E(\alpha)=\mult_\Delta E(\alpha)$.
\end{proof}

\begin{prop} \label{dimcancov} Let $\xid$ be a metrizable compactification pack and consider
$(\hx,\ceo)$. Let $n\in\Nset\cup\{0\}$. Then,
\begin{enumerate}[a) ]
\item $\asdim (\hx,\ceo)\leq n$.
\item For every uniform cover $\beta$ of $\hx$ there exists an open, locally finite and uniform
cover $\alpha$ of $\hx$ such that $\beta\prec\alpha$ and
$\mult\alpha\leq n+1$.
\end{enumerate}
are equivalent.
\end{prop}
\begin{proof}
Taking into account Corollary \ref{caractcanon}, it is obvious that b)
implies a). Let us see the reciprocal.

Consider a uniform cover $\beta$ of $\hx$. Let $E$ be an open,
symmetric and controlled neighborhood of the diagonal. Since $\asdim
(\hx,\ceo)\leq n$, there exists a uniform cover $\gamma$ of $\hx$
such that $\mult_{\Deelta(\beta)\circ E^2}(\gamma)\leq n+1$. Let
$\alpha=E\circ \Deelta(\beta)(\gamma)$. Then, $\mult_E\alpha=\mult_E
\left(E\circ \Deelta(\beta)(\gamma)\right)\leq\mult_{E^2}
\Deelta(\beta)(\gamma)\leq\mult_{\Deelta(\beta)\circ E^2}
(\gamma)\leq n+1$. Particulary, $\mult \alpha\leq\mult_E\alpha \leq
n+1$.

Clearly, $\alpha$ is uniform. To prove that $\alpha$ is open, fix
$V\in\alpha$ and take $U\in\gamma$ such that
$V=E(\Deelta(\beta)(U))$. Since
$V=\bigcup_{x\in\Deelta(\beta)(U)}E_x$ (see
(\ref{usefulidentities6})) and each $E_x$ is open, we have that $V$
is open. Moreover, for every $x\in \hx$,
$\mult_{E_x}\alpha\leq\mult_E\alpha\leq n+1<\infty$ and we get that
$\alpha$ is locally finite.

Finally, fix $W\in\beta$ and take $U\in \gamma$ with $W\in
U\neq\varnothing$. Then, by (\ref{usefulidentities1}), $W\subset
\bigcup_{\substack {W'\in \beta\\U\cap W'\neq
\varnothing}}W'=\beta(U)=\Deelta(\beta)(U)\subset E\circ
\Deelta(\beta)(U)\in\alpha$. Thus, $\beta\prec\alpha$.

Therefore, $\alpha$ is the desired cover.\end{proof}

\begin{teo}\label{dimasdimcanonicos} Let $\xid$ be a metrizable compactification pack and consider $(\hx,\ceo)$. Let $n\in\Nset\cup \{0\}$. Then,
\begin{enumerate}[a) ]
\item $\dim X\leq n$.
\item $\asdim(X,\ceo)\leq n+1$.
\item For every uniform cover $\beta$ of $\hx$ there exists a
canonical cover $\alpha\succ\beta$ which multiplicity is less than
or equal to $n+2$.
\item For every canonical cover $\beta$ of $\hx$ there exists canonical cover $\alpha\succ\beta$ which
multiplicity is less than or equal to $n+2$.\end{enumerate} are
equivalent.
\end{teo}

\begin{proof} Take into account Corollary \ref{caractcanon}. The equivalence between a) and b) is Grave's theorem.
b) implies c) due to Proposition \ref{dimcancov}. The implication c)
$\Rightarrow$ d) is obvious. d) implies b) because of Proposition
\ref{muchoscanonicos}.

Implications a)$\Rightarrow$ b), a)$\Rightarrow$ c) and
a)$\Rightarrow$ d) are also a consequence of Proposition
\ref{demgrave2}.
\end{proof}

\subsection{Canonical covers on special spaces}

Let $\xid$ be a metrizable compactification pack. If $\tx$ is the
Hilbert cube $Q$ or of the finite dimensional cube $[0,1]^n$,
Theorem \ref{dimasdimcanonicos} proves (B) partially, but we have to
show that the multiplicity of every canonical cover is greater than
$\dim X+2$, and not just for the bigger ones, as stated.

In general, just one canonical cover can not say anything about the
dimension of $X$. For example, suppose that $\hx$ is countable ---to
be more specific, let $Y$ be a metrizable compact set, let $\{y_n\}$
be a dense sequence in $Y$ and put $X=Y\times\{0\}$, $\hx=\left(
\bigcup_{n\in\Nset}\{y_1,\cdots,y_n\}\times\{\frac{1}{n}\}\right)$
and $\tx=X\cup\hx$, all of them with the topology induced by
$Y\times[0,1]$---. $\{\{x\}:x \in\hx\}$ is a canonical cover of
$\hx$ whose  multiplicity  is $1$, independently on the dimension of
$X$.

For proving (A), we need the special topological properties of the
Z-sets of $Q$ or $[0,1]^n$. By this reason, we will define the
\textit{cylindrical subsets}, a class of subsets with involve the
Z-sets of the Hilbert cube or the finite dimensional cube, which
have the properties we need to prove (A).

\begin{deff} A subset $X$ of a topological space $\tx$ is said to be
cylindrical if there exists an embedding $j:X\times
[0,1]\hookrightarrow\tx$ such that $j(x,0)=x$ for every $x\in X$.
\end{deff}

\begin{rem} If $X$ has a tubular neighborhood in $\tx$, then $X$ is
cylindrical in $\tx$, but the reciprocal is false (take for example
$X=\{0\}$ and $\tx=[-1,1]$).\end{rem}

\begin{deff} We say that a compactification pack $\xid$ is
cylindrical if $X$ is a cylindrical subset of $\tx$.\end{deff}

\begin{lemma}\label{conoenhc} Every Z-set of the Hilbert cube is a cylindrical subset.\end{lemma}

\begin{proof}
Anderson's theorem (see \cite{hc}, Theorem II-11.1 or
\cite{Anderson}) states that every homeomorphism between two Z-sets
of $Q$ can be extended to a homeomorphism of $Q$ onto itself.

Let $X$ be a Z-set of $Q$. Since $X\times  \{0\}$ and $X$ are Z-sets
of $Q\times[0,1]\approx Q$ and $Q$ respectively and
$g:X\times\{0\}\rightarrow X$, $(x,0)\rightarrow x$ is a
homeomorphism, there exists a homeomorphism
$h:Q\times[0,1]\rightarrow Q$ which extends $g$. Particulary,
$h|_{X\times[0,1]}:X\times[0,1]\rightarrow Q$ is and embedding and
$X$ is a cylindrical subset of $Q$.\end{proof}

\begin{rems}
If $h$ is the function  in Lemma \ref{conoenhc}'s proof and
$f=h^{-1}$, then $f:Q\rightarrow Q\times[0,1]$ is a homeomorphism
such that $f(X)\subset Q\times\{0\}$. Since the reciprocal is
obvious, we have an easy proof of the known result in infinite
dimensional topology: A closed subset $X$ of $Q$ is a Z-set of $Q$
if and only if there exists a homeomorphism $f:Q\rightarrow
Q\times[0,1]$ such that $f(X)\subset Q\times\{0\}$.

\item In $Q$, to be cylindrical is stronger than to be nowheredense
(see
$\left(\left\{\frac{1}{n}:n\in\Nset\right\}\cup\{0\}\right)\times Q$
in $[0,1]\times Q$) and weaker than to be a Z-set (see
$\left\{\frac{1}{2}\right\}\times Q$ in $[0,1]\times Q$).
\end{rems}

Using example VI 2 of \cite{hw}, it follows easily that:

\begin{lemma}\label{carazzsetcubito} Let $n\in\Nset$. The Z-sets of $[0,1]^{n}$ are precisely the
closed subsets of $[0,1]^n\bs (0,1)^n$\end{lemma}

\begin{lemma} Let $n\in\Nset$. Every Z-set of $[0,1]^n$ is a cylindrical
subset.
\end{lemma}

\begin{proof} Consider the copy $[-1,1]^n$ of $[0,1]^n$. Suppose that $X$ is a Z-set of $[-1,1]^n$. According with Lemma \ref{carazzsetcubito},
$X\subset [-1,1]^n\bs (-1,1)^n$. Consider the continuous maps
$j_1:X\times [0,1]\rightarrow X\times\left[\frac{1}{2},1\right]$,
$(x,t)\rightarrow \left(x,\frac{1-t}{2}\right)$ and
$j_2:X\times\left[\frac{1}{2},1\right]\rightarrow [-1,1]^n$,
$(x,t)\rightarrow t \cdot x$. It is easy to check that $j_2\circ
j_1:X\times[0,1]\rightarrow [-1,1]^n$ is an embedding such that
$j_2\circ j_1(x,0)=x$ for every $x\in X$.\end{proof}

\begin{lemma}\label{mejoraproducto01}
Let $X$ be a compact metric space with $\dim X\geq n<\infty$. Then,
there exists an $\epsil>0$ such that every open and finite cover
$\alpha$ of $X\times [0,1]$ with
\begin{enumerate}[a) ] \item $\mesh\{\pi_1(U):U\in\alpha\}<\epsil$
(where $\pi_1:X\times[0,1]\rightarrow X$ is the projection).
\item there does not exist any $U\in
\alpha$ which intersect both $X\times\{0\}$ and
$X\times\{1\}$.\end{enumerate} satisfies $\mult\alpha\geq n+2$.
\end{lemma}

\begin{proof}
Remember that if $X$ is a compact Hausdorff space, then $\dim
X\times[0,1]=\dim X+1$. It is a corollary of \cite{hur2}, page 194.
Also, from Theorem 7, Theorem 8 or Theorems 4-6 of \cite{mor}.

Consider on $X\times[0,1]$ the supremum metric. Since $\dim X\times
[0,1]= \dim X+1\geq n+1$, there exists $\epsil>0$ such that, for
every open cover $\beta$ of $X\times[0,1]$ with $\mesh\beta<\epsil$,
we have $\mult\beta \geq n+2$.

Let $\alpha$ be an open cover of $X$ satisfying properties a) and b)
for $\epsil$. Take $k\in\Nset$ with $\frac{1}{2k}<\epsil$. Let us
construct on $X\times [0,2k]$ a cover $\gamma_0$ like in the figure
\ref{figuracubiertarepetida}.

\begin{figurapng}{.37}{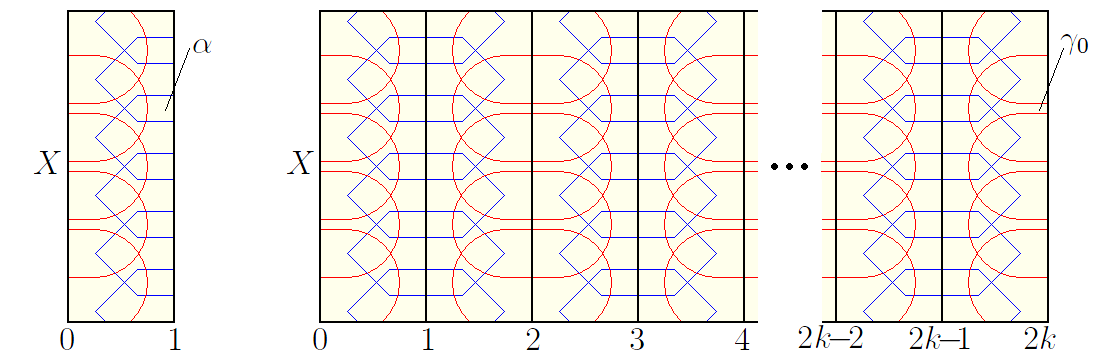}{Cover $\alpha$ of $X\times[0,1]$
and cover $\gamma_0$ of $X\times[0,2k]$}\label{figuracubiertarepetida}\end{figurapng}

Let $\alpha'$ be the symmetric cover of $\alpha$ on $X\times[0,1]$
given by $\alpha'=\{\phi(U):U\in\alpha\}$, where
$\phi:X\times[0,1]\rightarrow X\times[0,1]$ is the symmetry
$\phi(x,t)=(x,1-t)$. Pull forward the cover $\alpha$ to the
intervals $[2j,2j+1]$, for $j=0,\dots,k-1$, by means of the
translations $f_j:X\times \Rset\rightarrow X\times \Rset$,
$f_j(x,t)=(x,t+2j)$ and pull forward the cover $\alpha'$ to the
intervals $[2j+1,2j+2]$, for $j=0,\dots,k-1$, by means of the
translations $f'_j:X\times\Rset\rightarrow X\times \Rset$,
$f'_j(x,t)=(x,t+2j+1)$. Let $\gamma_0$ be the cover of $X\times
[0,2k]$ given by the union of those covers joining every pulled
$U\in\alpha$ which meets $X\times\{i\}$, for $i=1,\dots,2k-1$, with
their reflections on the pulled subsets of $\alpha'$.

More accurately,

$$\begin{array}{rl}\gamma_0=&\hspace{-8pt}\{f_j(U):U\in\alpha,0\leq j\leq
k-1,f_j(U)\cap X\times\{i\}=\varnothing\,\forall
i=1,\dots,2k-1\}\cup\\&\hspace{-8pt} \{f'_j\circ
\phi(U):U\in\alpha,0\leq j\leq k-1,f'_j\circ\phi(U)\cap
X\times\{i\}=\varnothing\,\forall
i=1,\dots,2k-1\}\cup\\&\hspace{-8pt}
 \{f_j(U)\cup f'_j\circ\phi(U):U\in\alpha,U\cap
X\times\{1\}\neq\varnothing,0\leq j\leq k-2\}\cup\\&\hspace{-8pt}
\{f'_j\circ\phi(U)\cup f_{j+1}(U):U\in\alpha,U\cap
X\times\{0\}\neq\varnothing,1\leq j\leq k-1\}\end{array}$$

Let $\gamma$ be the cover of $X\times[0,1]$ given by
$\gamma=\{\psi(U):U\in\gamma_0\}$, where $\psi:X\times
[0,2k]\rightarrow X\times [0,1]$ is the homothety
$\psi(x,t)=\left(x,\frac{1}{2k}t\right)$. It is easy to check that
$\gamma$ is an open cover of $X\times [0,1]$ with
$\mesh\gamma<\epsil$ which has the same multiplicity as $\alpha$. By
definition of $\epsil$, we get $\mult\alpha=\mult\gamma\geq
n+2$.\end{proof}

\begin{rem} We get an easy generalization of Lemma \ref{mejoraproducto01} by changing ``there exists $\epsil>0$'' by ``there exists an open and
finite cover of $X$ $\alpha_0$'' and ``$\mesh\pi_1(\alpha)<\epsil$''
by ``$\pi_1(\alpha)\prec\alpha_0$''. Hint: if $X$ and $Y$ are
compact spaces and $\alpha$ is an open and finite cover of $X\times
Y$, then there exists two open and finite covers $\beta_1$ and
$\beta_2$ of $X$ and $Y$ respectively such that $\{U\times
V:U\in\beta_1,V\in\beta_2\}\prec\alpha$.
\end{rem}

The following result is based on Theorem 5.9 of \cite{grv} (pag.
59):

\begin{lemma}\label{teodegrave} Let $X$ be a compact metric
space. Consider the compactification pack
$(X\times\{0\},X\times(0,1],X\times[0,1])$ and consider
$(X\times(0,1],\ceo)$. If $\alpha$ is an open and uniform cover of
$X\times(0,1]$, then $\dim X\leq \mult\alpha-2$. Particulary, it
happens if $\alpha$ is a canonical cover.
\end{lemma}

\begin{proof}
If $\mult\alpha=\infty$, the result is obvious. Suppose now that
$\mult\alpha=n<\infty$. To get a contradiction, assume that $\dim
X\geq n-1$. By Lemma \ref{mejoraproducto01}, there exists $\epsil>0$
such that every open cover $\beta$ of $X\times[0,1]$ satisfying
properties a) and b) of that lemma, also satisfies $\mult\beta\geq
n+1$.

Take into account the characterization of uniform covers of
Proposition \ref{recunif}. Consider on $X\times[0,1]$ the supremum
metric. Since $\alpha$ is uniform, there exists a compact subset of
$X\times(0,1]$ such that $\diam U<\epsil$ for every $U\in\alpha$
with $K\cap U=\varnothing$. Take $\delta_1>0$ such that $K\subset
X\times (\delta_1,0]$.

Since $\alpha$ is proper, $\alpha(X\times\{\delta_1\})$ is
relatively compact. Take $\delta_2>0$ such that
$\alpha(X\times\{\delta_1\})\subset X\times (\delta_2,1]$. Let
$\gamma_0=\{U\cap X\times[\delta_2,\delta_1]:U\in\alpha\}$ and
consider the cover $\gamma$ of $X\times[0,1]$ given by
$\gamma=\phi^{-1}(\gamma_0)$, where $\phi:X\times[0,1]\rightarrow
X\times[\delta_0,\delta_1]$ is the homeomorphism
$\phi(x,t)=\big(x,t\delta_2+(1-t)\delta_1\big)$.

Clearly, $\gamma$ is an open cover of $X\times[0,1]$ such that
$\mult\gamma\leq\mult\alpha\leq n$, $\mesh\pi_1(\gamma)<\epsil$ and
no $V\in \gamma$ intersects both $X\times\{0\}$ and $X\times\{1\}$.
By Lemma \ref{mejoraproducto01}, $\mult\gamma\geq n+1$, in
contradiction with $\mult\gamma=\mult\gamma_0\leq\mult\alpha\leq n$.
Hence, $\dim X\leq n-2$.\end{proof}

\begin{prop}\label{lemareccanonicosenq} Let $\xid$ be a metrizable
cylindrical compactification pack. If $\alpha$ is an open and
uniform cover of $(\hx,\ceo)$, then $\dim X\leq \mult\alpha-2$.
\end{prop}

\begin{proof}
Consider the compactification pack
$(X\times\{0\},X\times(0,1],X\times[0,1])$ and suppose
$\ceop=\ceo(X\times\{0\},X\times(0,1],X\times[0,1])$.

Let $\dd$ be a metric on $\tx$ and let
$j:X\times[0,1]\rightarrow\tx$ be an embedding such that $j(x,0)=x$
for every $x\in X$. Consider on $X\times[0,1]$ the metric $\dd'$
given by $\dd'(a,b)=\dd(j(a),j(b))$.

Let $\beta=j|_{X\times(0,1]}^{-1}(\alpha)$. From the continuity of
$j$ we get that $\beta$ is an open cover over $X\times[0,1]$ and,
from the inyectivity of $j$, that $\mult\beta\leq\mult\alpha$.

\begin{figurapng}{.18}{fig4}{A part of a uniform cover $\alpha$ or $(\hx,\ceo)$,}
\centerline{the induced cover in $j(X\times(0,1])$ and
$\beta=j|_{X\times(0,1]}^{-1}(\alpha)$}\end{figurapng}

Let us see that $\beta$ is uniform for $\ceop$. Let $\epsil>0$.
Since $\alpha$ is uniform for $\ceo$, there exists a compact subset
$K$ of $\hx$ such that $\dd(x,y)<\epsil$ whenever $(x,y)\in
\Deelta(\alpha)\bs K\times K$.

Let $K'=j^{-1}(K)$. Observe that $K'\subset j^{-1}(\hx)=X\times
(0,1]$. Moreover, $K'$ is compact, because it is a closed subset of
$X\times[0,1]$. Finally, for every $(a,b)\in \Deelta(\beta)\bs
K'\times K'$, we have that $(j(a),j(b))\in\Deelta(\alpha)\bs K\times
K$ and, consecuently, $\dd'(a,b)=\dd(j(a),j(b))<\epsil$.

Since $\beta$ is an open and uniform cover of $(X\times
(0,1],\ceop)$, Lemma \ref{teodegrave} shows that $\dim X\leq
\mult\alpha-2$.\end{proof}

\begin{prop} Let $\xid$ be a cylindrical metrizable compactification
pack (in particular, if $\tx$ is $Q$ or $[0,1]^n$, with $n\in\Nset$,
and $X$ is a Z-set of $\tx$). Consider $(\hx,\ceo)$.

Then, the following properties are equivalent:
\begin{enumerate}[a) ]
\item $\dim X\leq n$.
\item $\asdim(\hx,\ceo)\leq n+1$.
\item For every uniform cover $\beta$, there exists a canonical cover $\alpha$
such that $\beta\prec\alpha$ and $\mult\alpha \leq n+2$.
\item For every canonical cover $\beta$, there exists a canonical cover $\alpha$
such that $\beta\prec\alpha$  and $\mult\alpha \leq n+2$.
\item There exists a canonical cover $\alpha$ with $\mult\alpha \leq n+2$.

\item There exists an open and uniform cover $\alpha$ of $\xid$ with $\mult\alpha \leq n+2$.
\end{enumerate}
\end{prop}

\begin{proof}
The equivalences between a),b),c) and d) are given in Proposition
\ref{dimasdimcanonicos}. The implications d) $\Rightarrow$ e)
$\Rightarrow$ f) are obvious. f) implies a) because of Proposition
\ref{lemareccanonicosenq}.
\end{proof}

\subsection{Canonical covers on general spaces}

Let us start with the following lemma to get a corollary from
Proposition \ref{teodegrave}:

\begin{lemma}\label{lemaparticion} If $\alpha$ and
$\beta$ are families of subsets of a set $Z$ and there is a
surjective map $\phi:\alpha\twoheadrightarrow\beta$ such that
$U\supset \phi(U)$ for every $U\in\alpha$, then $\mult\beta\leq
\mult\alpha$.
\end{lemma}

\begin{proof}
$\mult\beta=\sup\{\#B:B\subset\beta,\bigcap_{V\in
B}V\neq\varnothing\}\leq\sup\{\#\phi^{-1}(B):B\subset\beta,\bigcap_{U\in
\phi^{-1}(B)}U\neq\varnothing\}\leq\sup\{\#A:A\subset\alpha,\bigcap_{U\in
A}U\neq\varnothing\}=\mult\alpha$.
\end{proof}

\begin{cor}\label{opencoversmeaning} Let $\xid$ be a cylindrical metrizable compactification pack and
consider $(\hx,\ceo)$. Then,
\begin{enumerate}[a) ]
\item If $\alpha$ is a uniform cover $\hx$ which has an open refinement,
then $\mult\alpha\geq\asdim (\hx,\ceo)+1$.
\item If $E$ is an open and controlled neighborhood of the diagonal of $\hx\times\hx$, then
$\mult_E\alpha\geq \asdim (\hx,\ceo)+1$ for every uniform cover
$\alpha$ of $\hx$.
\end{enumerate}
\end{cor}

\begin{proof}
To prove a), take an open refinement $\beta$ of $\alpha$. Let
$\calt$ be the topology of $\hx$ and consider the map
$\phi:\alpha\rightarrow \calt$ given by
$\phi(V)=\bigcup_{\substack{U\in\beta\\U\subset V}}U$. Observe that
$U\supset\phi(U)$ for each $U$. Let $\gamma=\phi(\alpha)$. By Lemma
\ref{lemaparticion}, $\mult\gamma\leq\mult\alpha$.

Since $\gamma\prec\alpha$, we have that $\gamma$ is uniform.
Moreover,
$$\bigcup_{W\in\gamma}W=\bigcup_{V\in\alpha}\phi(V)=\bigcup_{V\in\alpha}\bigcup_{\substack{U\in\beta\\U\subset
V}}U=\bigcup_{U\in\beta}U=\hx$$ Since $\gamma$ is an open and
uniform cover of $\hx$, Proposition \ref{teodegrave} and Grave's
theorem show that $\mult\alpha\geq\mult \gamma\geq\dim X+2=\asdim
(\hx,\ceo)+1$.

Now, let us see b). $E(\alpha)$ is uniform because $E$ is controlled
and $\alpha$, uniform. By (\ref{usefulidentities6}),
$E(\alpha)=\{E(V):V\in\alpha\}=\{\bigcup_{x\in V}E_x:V\in\alpha\}$.
Since each $E_x$ is open, we get that $E(\alpha)$ is open and, since
$\alpha\prec E(\alpha)$, that $E(\alpha)$ is a cover of $\hx$.

From Lemmas \ref{cambiomult} and \ref{lemareccanonicosenq} and
Grave's theorem, we get $\mult_E\alpha\geq\mult E(\alpha)\geq\dim
X+2\geq\asdim (\hx,\ceo)+1$. \end{proof}

Let $\xid$ be a cylindrical and metrizable compactification pack and
consider $(\hx,\ceo)$. Corollary \ref{opencoversmeaning} means that,
from the asymptotic dimensional point of view, an open and uniform
cover $\alpha$ or a controlled and open neighborhood of the diagonal
are big enough.

This observation suggests a question. In the general case when
$\xid$ is not necessary cylindrical, when is a cover $\alpha$ or a
controlled set $E$ big enough from the asymptotic dimensional point
of view? The following results will answer this question.

\begin{prop}\label{propsh} Let $\xid$ be a metrizable compactification pack and
suppose that $\dd$ is a metric on $\tx$. If $k=\sup_{x\in
\tx}\dd(x,X)$, then the map $h:(0,k]\rightarrow (0,k]$,
$t\rightarrow\sup_{x\in X}\dd(x,\tx\bs\BB(X,t))$ is increasing and
satisfies $\lim_{t\rightarrow 0}h(t)=0$.\end{prop}

\begin{proof}
If $t\leq t'$, then $\tx\bs \BB(X,t)\supset\tx\bs\BB(X,t')$ and, for
every $x\in X$, $\dd(x,\tx\bs\BB(X,t))\leq \dd(x,\tx\bs\BB(X,t'))$.
Taking supreme over $x$, we get  $h(t)\leq h(t')$.

Fix $\epsil>0$. Since $X$ is compact and $X\subset\bigcup_{x\in
X}\BB\left(X,\frac{\epsil}{2}\right)$, there exists
$x_1,\dots,x_r\in X$ such that
$X\subset\bigcup_{j=1}^r\BB\left(x_j,\frac{\epsil}{2}\right)$. For
every $j$, pick a point $y_j\in
\hx\cap\BB\left(x_j,\frac{\epsil}{2}\right)$. Let
$\delta=\min_{1\leq j\leq r}\dd(y_j,X)$.

Fix $t< \delta$. For every $x\in X$, there is $j=1,\dots,r$ such
that $x\in\BB\left(x_j,\frac{\epsil}{2}\right)$. Observe that
$y_j\in \tx\bs\BB(X,t)$ and thus, $\dd(x,\tx\bs\BB(X,t))\leq
\dd(x,y_j)\leq\dd(x,x_j)+\dd(x_j,y_j)<\frac{\epsil}{2}+\frac{\epsil}{2}=\epsil$.
Taking supreme over $x$, we get $h(t)\leq\epsil$.\end{proof}

The following proposition implies Corollary
\ref{variasvecesdescubierto}. This corollary has been proved
independently by Grave in \cite{grv} or \cite{grv2}, by us in
\cite{moreno} and \cite{moreno2} and by Mine and Yamashita in
\cite{miya}. We add this proposition here because we need the
explicit functions used there. Moreover, we get an easy proof of
Corollary \ref{variasvecesdescubierto}.

\begin{prop}\label{fgequivcoarse} Let $\xid$ be a metrizable compactification pack.
Consider the compactification pack
$(X\times\{0\},X\times(0,1],X\times[0,1])$ and the coarse structures
$\ceo=\ceo\xid$ and
$\ceop=\ceo(X\times\{0\},X\times(0,1],X\times[0,1])$. Let $\dd_0$ be
a metric on $\tx$. Consider the metric $\dd=\frac{1}{k}\dd_0$ on
$\tx$, where $k=\sup_{x\in\tx}\dd_0(x,X)$.

Then, there exist an $f:(\hx,\ceo)\rightarrow(X\times(0,1],\ceop)$
and a $g:(X\times(0,1],\ceop)\rightarrow(\hx,\ceo)$ satisfying
\begin{itemize}
\item For every $x\in
\hx$ $f(x)=(z,t)$, with $t=\dd(x,X)$, $z\in X$ and $\dd(x,z)=t$.

\item For every $(z,t)\in X\times(0,1]$, $g(z,t)=y$ with $y\in \tx\bs\BB(X,t)$ and
$\dd(y,z)=\dd(z,\tx\bs\BB(X,t))$.\end{itemize} in which case, they
are coarse equivalences, the one inverse of the other.
\end{prop}

\begin{proof} Since $\sup_{x\in\tx}\dd(x,X)=1$, such $f$ and
$g$ do exist. To check their coarse equivalentness, it suffices to
show that they satisfice property c) of Proposition \ref{criterio}.
Consider on $X\times [0,1]$ the metric
$\dd'((x,t),(z,s))=\dd(x,z)+|t-s|$.

Fix $E\in\ceo$ and let us see that $f\times f(E)\in\ceop$. Let
$\epsil>0$ and suppose that $K$ is a compact subset of $\hx$ such
that $\dd(x,x')<\frac{\epsil}{5}$ whenever $(x,x')\in E\bs K\times
K$.

Let $\delta=\min\{\frac{\epsil}{5},\dd(K,X)\}$ and consider
$K'=X\times[\delta,1]$. Pick $((z,t),(z',t'))\in f\times f(E)\bs
K'\times K'$ and take $(x,x')\in E$ such that $f(x)=(z,t)$ and
$f(x')=(z',t')$. Suppose, without loss of generality, that $t\leq
t'$.

Neither $(z,t)\in K'$ nor $(z',t')\in K'$, that is, either
$t<\delta$ or $t'<\delta$. Then, $\dd(x,X)=t<\delta\leq\dd(K,X)$ and
thus, $x\not\in K$. Hence $(x,x')\in E\bs K\times K$ and
$\dd(x,x')<\frac{\epsil}{5}$. Moreover,
$t'=\dd(x',X)\leq\dd(x',z)\leq\dd(x',x)+\dd(x,z)=\dd(x',x)+t<\dd(x',x)+\delta$.
Therefore,
$$\xymatrix@R=.1mm{
\dd'((x,t),(z,s))=\dd(z,z')+|t-t'|\leq\\
\dd(z,x)+\dd(x,x')+\dd(x',z')+t'-t=\\
t+\dd(x,x')+t'+t'-t=\dd(x,x')+2t'\leq\\
3\dd(x,x')+2\delta<3\frac{\epsil}{5}+2\frac{\epsil}{5}=\epsil}$$ and
we get $f\times f(E)\in\ceop$.

Fix $F\in\ceop$ and let us see that $(f\times f)^{-1}(F)\in\ceo$.
Let $\epsil>0$ and suppose that $K'$ is a compact subset of $X\times
(0,1]$ such that $\dd'((z,t),(z',t'))<\frac{\epsil}{3}$ whenever
$((z,t),(z',t'))\in F\bs K'\times K'$. Take $\delta_0$ such that
$K'\subset X\times[\delta_0,1]$. Put $\delta=\min
\{\delta_0,\frac{\epsil}{3}\}$ and put
$K=\tx\bs\BB\left(X,\delta\right)$. Pick $(x,x')\in (f\times
f)^{-1}(F)\bs K\times K$ and take $((z,t),(z',t'))\in F$ such that
$f(x)=(z,t)$ and $f(x')=(z',t')$. Suppose, without loss of
generality, that $t\leq t'$.

Neither $x\in K$ nor $x'\in K$, that is, either $t=\dd(x,X)<\delta$
or $t'=\dd(x',X)<\delta$. Then, $t<\delta\leq\delta_0$ and thus
$(z,t)\not\in K'$. Hence, $((z,t),(z',t'))\in E\bs K'\times K'$ and
$\dd'((z,t),(z',t'))<\frac{\epsil}{3}$. Therefore,
$$\xymatrix@R=.1mm{\dd(x,x')\leq
\dd(x,z)+\dd(z,z')+\dd(z',x')=\\
t+\dd(z,z')+t'=2t+\dd(z,z')+|t-t'|<\\
2\delta+\dd'((z,t),(z',t'))<\frac{2\epsil}{3}+\frac{\epsil}{3}=\epsil}$$
and we get $(f\times f)^{-1}(F)\in\ceo$.

Let $G=\{((z,t),f\circ g(z,t)):(z,t)\in X\times (0,1]\}$ and let us
see that $G\in\ceop$. Fix $\epsil>0$. Consider the function
$h:(0,1]\rightarrow(0,1]$, $t\rightarrow\sup_{x\in
X}\dd(x,\tx\bs\BB(X,t))$. By Proposition \ref{propsh},
$\lim_{t\rightarrow 0}h(t)=0$, so there exists $\delta>0$ such that
$h(t)<\frac{\epsil}{3}$ when $t<\delta$.

Let $K=X\times[\delta,1]$ and pick $((x,t),(z,s))\in G\bs K\times
K$. Then, either $t<\delta$ or $s<\delta$. Observe that
$(z,s)=f\circ g(x,t)$. Put $y=g(x,t)$, so we have $(z,s)=f(y)$.
Since $y\in\tx\bs\BB(X,t)$, it follows that $s=\dd(y,X)\geq t$ and
thus, $t<\delta$. Therefore,

$$\xymatrix@R=.1mm{\dd'((x,t),(z,s))=\dd(x,z)+|t-s|\leq\\
\dd(x,y)+\dd(y,z)+s-t=\dd(x,y)+2\dd(y,z)-t<\\
\dd(x,y)+2\dd(y,X)\leq 3\dd(x,y)=\\
3\dd(x,\tx\bs\BB(X,t))\leq 3h(t)<3\frac{\epsil}{3}=\epsil}$$
and we get $G\in\ceop$. Therefore, $f$ is a coarse equivalence and
$g$ is its coarse inverse.
\end{proof}

\begin{cor}\label{variasvecesdescubierto} If $(X_1,\hx_1,\tx_1)$ and $(X_2,\hx_2,\tx_2)$ are two metrizable
compactification packs such that $X_1$ and $X_2$ are homeomorphic,
then $(\hx_1,\mathcal{E}_0^1)$ and $(\hx_2,\mathcal{E}_0^2)$ are
coarse equivalent where, for $i=1,2$,
$\mathcal{E}_0^i=\ceo(X_i,\hx_i,\tx_i)$.\end{cor}

\begin{proof} For $i=1,2$, consider the compactification pack
$(X_i\times\{0\},X_i\times(0,1],X_i\times[0,1])$ and suppose
$\mathcal{E}_0^{\prime
i}=\ceo(X_i\times\{0\},X_i\times(0,1],X_i\times[0,1])$.

Let $h:X_1\rightarrow X_2$ be a homeomorphism and let $\dd$ be a
metric on $X_1$. Consider on $X_1\times [0,1]$ and $X_2\times[0,1]$
the metrics $\dd_1$ and $\dd_2$ respectively, given by:
$$\xymatrix@R=.4mm{\dd_1((x,t),(y,s))=\dd(x,y)+|t-s|\\
\dd_2((x,t),(y,s))=\dd(h^{-1}(x),h^{-1}(y))+|t-s|}$$

Using $\dd_1$ and $\dd_2$, it is easy to check that the map
$h':X_1\times(0,1]\rightarrow X_2\times(0,1]$, $(x,t)\rightarrow
(h(x),t)$ satisfies property d) of Proposition \ref{criterio} and
hence, $h'$ is a coarse equivalence. Finally, from Proposition
\ref{fgequivcoarse}, we get
$$(\hx_1,\mathcal{E}_0^1)\approx(X_1\times(0,1],\mathcal{E}_0^{\prime
1})\approx(X_2\times(0,1],\mathcal{E}_0^{\prime
2})\approx(\hx_2,\mathcal{E}_0^2)$$\end{proof}

\begin{prop}\label{caractc0} Let $\xid$ be a compactification pack, let $\dd$ be a metric on $\tx$, let $E\subset\hx\times\hx$ and let $k=\sup_{x\in\tx}\dd(x,X)$. Then,
$E\in\ceo$ if and only if there exists
$\phi:(0,k]\rightarrow\Rset^+$ with $\lim_{t\rightarrow 0}\phi(t)=0$
such that
$$E\subset\{(x,y)\in
X:\dd(x,y)<\phi(\min\{\dd(x,X),\dd(y,X)\})\}$$
\end{prop}

\begin{proof}
Suppose such $\phi$ exists. Fix $\epsil>0$ and take $\delta>0$ such
that $\phi(t)<\epsil$ for every $t<\delta$. Let $K=\tx\bs
\BB(X,\delta)$ and pick $(x,y)\in E\bs K\times K$. Then, neither
$x\in K$ nor $y\in K$. In any case,
$\min\{\dd(x,X),\dd(y,X)\}<\delta$. Hence,
$\dd(x,y)\leq\phi(\min\{\dd(x,X),\dd(y,X)\})<\epsil$ and $E\in\ceo$.

Now, assume that $E\in\ceo$. For every $t\in (0,k]$, let $K_t=\tx\bs
\BBC(X,t)$ and let $\phi(t)=t+\sup\{\dd(x,y):(x,y)\in E\bs K_t\times
K_t\}$ (if $E\bs K_t\times K_t$ is empty, put $\phi(t)=t)$.  Let us
see that $\lim_{t\rightarrow 0}\phi(t)= 0$. Fix $\epsil>0$, consider
a compact subset $K$ of $\hx$ such that $\dd(x,y)<\frac{\epsil}{2}$
whenever $(x,y)\in E\bs K\times K$ and set $\delta=\dd(x,K)$. Fix
$t<\delta$ and pick $(x,y)\in E\bs K_t\times K_t$. Since $K\subset
K_t$ and hence $(x,y)\in E\bs K\times K$, we have that
$\dd(x,y)<\frac{\epsil}{2}$, $t<\frac{\epsil}{2}$ and
$\phi(t)<\epsil$.

Pick a point $(x,y)\in E$ and put $t_0=\min\{\dd(x,X),\dd(y,Y)\}$.
Since $(x,y)\in E\bs K_{t_0}\times K_{t_0}$, we have that $\dd(x,y)<
\phi(t_0)=\phi(\min\{\dd(x,X),\dd(y,X)\})$.
\end{proof}

\begin{prop}\label{caractonbh} Let $\xid$ be a compactification pack,
let $\dd$ be a metric on $\tx$, let $E\subset\hx\times\hx$ and let
$k=\sup_{x\in\tx}\dd(x,X)$. Then, $E$ is a neighborhood of the
diagonal if and only if there exists an increasing function
$\lambda:(0,k]\rightarrow\Rset^+$ such that
$$\{(x,y)\in
X:\dd(x,y)<\lambda(\min\{\dd(x,X),\dd(y,X)\})\}\subset E$$
\end{prop}

\begin{proof}
Suppose such $\lambda$ exists. Let $\lambda_0:(0,k]\rightarrow
\Rset^+$ be an increasing and continuous map such that
$\lambda_0(t)\leq\lambda(t)$ for every $t$. For example, we may define
$\lambda_0$ as follows: set
$\lambda_0\left(\frac{k}{n}\right)=\lambda\left(\frac{k}{n+1}\right)$
for every $n\in\Nset$ and extend $\lambda_0$ linearly to every
interval $\left[\frac{k}{n},\frac{k}{n+1}\right]$.

Let $F=\psi^{-1}(\Rset^+)$, where $\psi$ is the continuous function
$\psi:\hx\times\hx\rightarrow\Rset$, $(x,y)\rightarrow
\lambda_0(\min\{\dd(x,X),\dd(y,X)\})-\dd(x,y)$. Then, $F$ is an open
subset of $\hx\times\hx$ containing the diagonal such that $F\subset
\{(x,y)\in X:\dd(x,y)<\lambda(\min\{\dd(x,X),\dd(y,X)\})\}\subset
E$.

Now, suppose $E$ is a neighborhood of the diagonal. Consider the
supremum metric $\dd_\infty$ on $\tx\times\tx$. Take an open subset
$F\subset E$ containing the diagonal. By the closeness of
$\hx\times\hx\bs F$, we may define the map $f:X\rightarrow \Rset^+$,
$x\rightarrow\dd_\infty((x,x),\hx\times\hx\bs F)$ and it is
continuous. For every $k\in(0,k]$, $K_t$ is compact and then, $f$
has a minimum in $K_t$, so we may define the map
$\lambda:(0,k]\rightarrow \Rset^+$, $t\rightarrow \min_{t\in
K_t}f(t)$.

For every $t\leq t'$, $K_t\supset K_{t'}$, hence
$\lambda(t)\leq\lambda(t')$ and we get that $\lambda$ is increasing.
Let $(x,y)\in\hx\times\hx$ be such that
$\dd(x,y)<\lambda(\min\{\dd(x,X),\dd(y,X)\})$. Suppose, without loss
of generality, that $\dd(x,X)\leq\dd(y,X)$ and put $t=\dd(x,X)$.
Since $x\in K_t$, we have
$$\dd_\infty((x,x),(x,y))<\lambda(t)\leq\dd_\infty((x,x),\hx\times\hx\bs
F)$$ Then, $(x,y)\not\in\hx\times\hx\bs F$ and we conclude that
$(x,y)\in F\subset E$.\end{proof}

\begin{lemma}\label{multfmenos1} Let $f:Z\rightarrow Z'$ be a map between two sets, let $E\subset Z\times
Z$, and let $\alpha$ be a family of subsets of $Z'$. Then, $\mult_E
f^{-1}(\alpha)\leq\mult_{f\times f(E)}\alpha$.\end{lemma}

\begin{proof}
Fix $x\in Z$ and take $U\in\alpha$ such that $f^{-1}(U)$ meets
$E_x$. Then, by (\ref{usefulidentities5}), $\varnothing\neq
f(f^{-1}(U)\cap E_x)=U\cap f(E_x)\subset U\cap \bigcup_{z\in
f^{-1}(f(x))}f(E_z)=U\cap (f\times f(E))_{f(x)}$.

Hence, $\mult_{E_x}f^{-1}(\alpha)=\#\{f^{-1}(U):U\in\alpha,f^{-1}(U)\cap
E_x\neq\varnothing\}\leq\#\{U:U\in\alpha,U\cap (f\times
f(E))_{f(x)}\neq\varnothing\} \leq\mult_{(f\times
f(E))_{f(x)}}\alpha\leq\mult_{f\times f(E)}\alpha$. Taking supreme over $x$ we
get the inequality.\end{proof}

\begin{prop}\label{equivalenteaabierto} Let $\xid$ be a metrizable compactification pack, let
$\dd$ be a metric on $\tx$, let $k=\sup_{x\in \tx}\dd(x,X)$ and let
$\lambda:(0,k]\rightarrow\Rset^+$ be an increasing and continuous
function such that $\lim_{t\rightarrow 0}\lambda(t)=0$.

Consider the maps $h:\Rset^+\rightarrow(0,k]$ and
$\phi:(0,k]\rightarrow\Rset^+$ given by:
$$h(t)=\left\{\begin{array}{ll} \sup_{x\in X}\dd(x,\tx\bs\BB(X,t)&\textrm{if }t\leq
k\\k &\textrm{if }t\geq k\end{array}\right.$$ $$\phi(t)=
 h(t)+\lambda(t)+h(t+\lambda(t))$$ and consider the set:
$$E_{\dd,\lambda}=\{(x,y)\in
\hx\times\hx:\dd(x,y)<\phi(\min\{\dd(x,X),\dd(y,X)\})\}$$

Then, $E_{\dd,\lambda}$ is a controlled subset of $(\hx,\ceo)$ such
that $\mult_{E_{\dd,\lambda}}\alpha\geq\dim X+2$ for every uniform
cover $\alpha$ of $(\hx,\ceo)$\end{prop}

\begin{proof}
By Proposition \ref{propsh}, $\phi(t)\rightarrow 0$ when
$t\rightarrow 0$. Hence, from Proposition \ref{caractc0}, we get
$E_{\dd,\lambda}\in\ceo$. Let $\alpha$ be a uniform cover of
$(\hx,\ceo)$.

Consider the compactification pack
$(X\times\{0\},X\times(0,1],X\times[0,1])$ and the coarse space
$(\hx,\ceop)$, where
$\ceop=\ceo(X\times\{0\},X\times(0,1],X\times[0,1])$. Consider on
$\tx$ the metric $\overline{\dd}=\frac{1}{k}\dd$ and, on
$X\times[0,1]$, the metric
$\overline{\dd}_\infty((x,t),(x',t'))=\max\{\dd(x,x'),|t-t'|\}$.

By Proposition \ref{fgequivcoarse}, there is a coarse equivalence
$g:X\times(0,1]\rightarrow \hx$ such that for every $(x,t)$,
$g(x,t)=y$ where $y\in\tx\bs\BB_{\overline{\dd}}(X,t)$ is such that
$\overline{\dd}(x,y)=\overline{\dd}(x,\tx\bs\BB_{\overline{\dd}}(X,t))$.
Consider the function $\overline{\lambda}:(0,1]\rightarrow \Rset^+$
such that for every $t$,
\begin{equation}\label{eqjjjalse}
\overline{\lambda}(t)=\frac{1}{k}\lambda(kt)
\end{equation} and consider the set $F_0=\big\{((x,t),(x',t'))\in
(X\times
(0,1])^2:\overline{\dd}_\infty((x,t),(x',t'))<\overline{\lambda}(\min\{t,t'\})\big\}$.

Since, for each $(x,t)$, we have
$t=\overline{\dd}_\infty((x,t),X\times\{0\})$, Proposition
\ref{caractc0} shows $F_0\in\ceo$ and Proposition \ref{caractonbh},
that $F_0$ is a neighborhood of the diagonal. Let $F\subset F_0$ be
a controlled and open neighborhood of the diagonal and let
$\beta=F(g^{-1}(\alpha))$.

Let us see that:
\begin{enumerate}[$\,\,\,$i) ]
\item  $\beta$ is
an open and uniform cover of $(X\times(0,1],\ceop)$.
\item $g\times
g(F)\subset E_{\dd,\lambda}$.
\end{enumerate}

For each $V\in \beta$, $V=F(U)=\bigcup_{x\in U}F_x$ for some $U\in
g^{-1}(\alpha)$. Since each $F_x$ is open, $V$ is open. Moreover,
$g^{-1}(\alpha)\prec\beta$ and hence $\beta$ is a cover of
$X\times(0,1]$. Finally, by Corollary \ref{corofequivalpha},
$\beta=F(g^{-1}(\alpha))$ is uniform for $\ceop$ and we get i).

Let us consider:
\begin{itemize}\item the map
$\overline{h}:\Rset^+\rightarrow(0,1]$ such that
$\overline{h}(t)=\sup_{x\in
X}\overline{\dd}(x,\tx\bs\BB_{\overline{\dd}}(X,t))$ when $t\leq 1$
and $\overline{h}(t)=1$ when $t\geq 1$ \item the map
${\overline{\phi}}:(0,1]\rightarrow\Rset^+$, given by
$\overline{\phi}(t)=\overline{h}(t)+\overline{\lambda}(t)+\overline{h}(t+\overline{\lambda}(t))$
\item the set
$E_{\overline{\dd},\overline{\lambda}}=\{(x,y)\in
X:\overline{\dd}(x,y)<{\overline{\phi}}(\min\{\overline{\dd}(x,X),\overline{\dd}(y,X)\})\}$\end{itemize}

By Proposition \ref{propsh}, $\overline{h}$ and $\overline{\phi}$
are increasing. It is easy to check that, for every $t$,
$\overline{h}(t)=\frac{1}{k}h(kt)$ and
$\overline{\phi}(t)=\frac{1}{k}\phi(kt)$. Using those equalities and
(\ref{eqjjjalse}), it follows easily that
$E_{\overline{\dd},\overline{\lambda}}=E_{\dd,\lambda}$. Then, to
prove ii), it suffices to show that $g\times g(F)\subset
E_{\overline{\dd},\overline{\lambda}}$.

Pick $(y,y')\in g\times g(F)$ and take $((x,t),(x',t'))\in F$ such
that $g(x,t)=y$ and $g(x',t')=y'$. Suppose, without loss of
generality, that $t\leq t'$. Observe that $t'=t+|t-t'|<
t+\overline{\dd}_\infty((x,t),(x',t'))\leq t+\overline{\lambda}(t)$.
Then:
$$\overline{\dd}(y,y')\leq\overline{\dd}(y,x)+\overline{\dd}(x,x')+\overline{\dd}(x',y')\leq$$
$$\overline{\dd}(x,\tx\bs\BB_{\overline{\dd}}(X,t))+\overline{\dd}_\infty((x,t),(x',t'))+\overline{\dd}(x',\tx\bs\BB_{\overline{\dd}}(X,t'))<$$
\begin{equation}\label{eqfadeas3w}\overline{h}(t)+\overline{\lambda}(t)+\overline{h}(t')\leq
\overline{h}(t)+\overline{\lambda}(t)+\overline{h}(t+\overline{\lambda}(t))=\overline{\phi}(t)\end{equation}

Since $y\in\tx\bs\BB_{\overline{\dd}}(X,t)$ and
$y'\in\tx\bs\BB_{\overline{\dd}}(X,t')$, it follows that
$\overline{\dd}(y,X)\geq t$ and $\overline{\dd}(y',X)\geq t'\geq t$.
Since ${\overline{\phi}}$ is increasing, from (\ref{eqfadeas3w}) we
get:
$$\overline{\dd}(y,y')<{\overline{\phi}}(t)\leq
{\overline{\phi}}(\min\{\overline{\dd}(y,X),\overline{\dd}(y',X)\})$$
Therefore, $(y,y')\in E_{\overline{\dd},\overline{\lambda}}$ and we
get ii).

Applying i), ii), Proposition \ref{lemareccanonicosenq} and Lemmas
\ref{cambiomult} and \ref{multfmenos1}, we conclude:
$$\xymatrix@R=.4mm{\dim X+2\leq \mult\beta=\mult
F(g^{-1}(\alpha))\leq\\
\mult_F g^{-1}(\alpha)\leq\mult_{g\times
g(F)}\alpha\leq\mult_{E_{\dd,\lambda}}\alpha}$$
\end{proof}

\begin{cor}\label{coroequivabto} Let $\xid$ be a compactification pack and
consider $(\hx,\ceo)$. Let $E$ be the controlled set
$E_{\dd,\lambda}$ of Proposition \ref{equivalenteaabierto}. Then:

\begin{enumerate}[a) ]
\item $\mult_E\alpha\geq \asdim (\hx,\ceo)+1$ for every uniform cover
$\alpha$ of $\hx$.

\item $\mult\alpha\geq\asdim(\hx,\ceo)+1$ for every uniform cover $\alpha$ of $\hx$ such that
$\Kappa(E)\prec\alpha$.
\end{enumerate}\end{cor}

\begin{proof}
Let $\alpha$ be a uniform cover of $\hx$. By Proposition
\ref{equivalenteaabierto} and Grave's theorem, $\mult_E\alpha\geq
\dim X+2=\asdim (\hx,\ceo)+1$, so we get a).

To see b), suppose $\Kappa(E)\prec\alpha$. For every $U\in\alpha$,
let $V_U=\{x\in \hx:E_x\subset U\}$ and let $\gamma=\{V_U:U\in
\alpha\}$. Observe that, by (\ref{usefulidentities6}),
$E(V_U)=\bigcup_{x\in V_U}E_x\subset U$ for every $U$. Hence,
$\gamma\prec\alpha$ and we get that $\gamma$ is uniform. For all
$x\in \hx$, there is $U\in\alpha$ such that $E_x\subset U$ and,
consecuently, $x\in V_U$. Then, $\gamma$ is a cover of $\hx$.

Fix $x\in \hx$ and let $U\in\alpha$ be such that $E_x\cap
V_U\neq\varnothing$. Let $y\in E_x\cap V_U$. Since $E$ is symmetric,
$x\in E_y$ and, since $y\in V_U$, we have that $E_y\subset U$. Then,
$x\in U$. Hence, $\mult_{E_x}\gamma=\#\{V_U:U\in\alpha,E_x\cap
V_U\neq\varnothing\}\leq\#\{U:U\in\alpha,x\in U\}=\mult_x\alpha$.
Taking supreme over $x$ and applying a) we get:
$$\mult_E\gamma\leq \mult\alpha\leq \asdim (\hx,\ceo)+1$$
\end{proof}

Corollary \ref{coroequivabto} means that, from the point of view of
the asymptotic dimension, the set $E_{\dd,\lambda}$ of Proposition
\ref{equivalenteaabierto} and the cover $\Kappa(E_{\dd,\lambda})$
are big enough.

\begin{prop} Let $\xid$ be a metrizable compactification pack and consider $(\hx,\ceo)$.
Then, the following are equivalent:
\begin{enumerate}[a) ]
\item $\dim X\leq n$.
\item $\asdim(\hx,\ceo)\leq n+1$.
\item For every uniform cover $\beta$ there exists a canonical cover
$\alpha$ such that $\beta\prec\alpha$ and $\mult\alpha \leq n+2$.
\item For every canonical cover $\beta$ there exists a canonical cover $\alpha$
such that  $\beta\prec\alpha$ and $\mult\alpha \leq n+2$.
\item There exists a uniform cover $\alpha$ (a
canonical cover, respectively) such that $\mult_E\alpha\leq n+1$,
where $E$ is the subset $E_{\dd,\lambda}$ of Proposition
\ref{equivalenteaabierto}.
\end{enumerate}
\end{prop}
\begin{proof} It is a consequence of Corollary \ref{caractcanon} and Propositions \ref{dimasdimcanonicos} and
\ref{equivalenteaabierto}.\end{proof}

\section{An easier proof of Grave's theorem}

Let $\xid$ be a compactification pack. In order to prove Grave's
theorem, we have to see that:
\begin{equation}\label{grave1}
\asdim (\hx,\ceo)\geq \dim X+1
\end{equation}
\begin{equation}\label{grave2}
\asdim (\hx,\ceo)\leq \dim X+1
\end{equation}

\begin{proof}[Proof 1 of (\ref{grave1})] It is a consequence of Proposition \ref{equivalenteaabierto}.\end{proof}

But it is not the natural way to prove (\ref{grave1}). The natural
way (more or less the way used by Grave applying other result
instead of Proposition \ref{teodegrave}) is the following:

\begin{proof}[Proof 2 of (\ref{grave1})] Consider $(X\times(0,1],\ceop)$, where
$\ceop=\ceo(X\times\{0\},X\times(0,1],X\times[0,1])$).

Since $X$ and $X\times \{0\}$ are homeomorphic, by Corollary
\ref{variasvecesdescubierto}, it follows that $(\hx,\ceo)$ and
$(X\times(0,1],\ceop)$ are coarse equivalent. Therefore, $\asdim
(\hx,\ceo)=\asdim (X\times(0,1],\ceop)\geq\dim X+1$, where the last
inequality is given by Propositions \ref{muchoscanonicos} and
\ref{teodegrave}.
\end{proof}

In Theorem 2.9 of \cite{cc} (pag. 3712), the authors defined
a canonical cover similar to $\alpha(\{\beta_n\},\{W_n\})$ of Lemma
\ref{basicodefrec}. They used it to prove that if $X$ has finite
dimension, then there exists a canonical cover with finite
multiplicity. They bounded the multiplicity of that canonical cover
by $2\dim X+2$.

The construction of that canonical cover induces an implicit problem
in the Ttopological Dimension Theory: To decrease the bound of the
multiplicite of that canonical cover to the minimum (that is $\dim
X+2$) we need the sequence $\{\beta_n\}$ to satisfice some special
dimensional properties. Does this special sequence exists?

We solved that topological problem in dimension theory in
\cite{moreno} ---we quote it in Theorem \ref{teomoreno}---. With
this techniques, we will be able to give a new proof of
(\ref{grave2}) and, at the same time, define a canonical cover with
minimal multiplicity.

\begin{deff}\label{defmultcomun} Let $\alpha_1,\dots,\alpha_m$ be families of subsets of a set $Z$.
The common multiplicity of $\alpha_1,\dots,\alpha_m$ is:
$$\mult(\alpha_1,\cdots,\alpha_m)=\sup_{x\in Z}
\mult_x\alpha_1+\cdots+\mult_x\alpha_m=$$
$$\sup\left\{\sum_{i=1}^m\#A_i:A_i\subset\alpha_i\,\forall i,
\bigcap_{i=1}^m\bigcap_{U\in A_i} U\neq\varnothing\right\}$$
\end{deff}

\begin{rem}
$\mult(\alpha_1,\cdots,\alpha_m)$ is a multiplicity greater than or
equal to $\mult(\alpha_1\cup\dots\cup\alpha_m)$ which is equal when
$\alpha_1,\dots,\alpha_r$ are pairwise disjoint.
\end{rem}

\begin{prop}\label{propiedadesv} Let $(\tx,\dd)$ be a metric space and consider $X\subset\tx$.
Consider the topology $\calt$ of $\tx$ attached to $\dd$ and the map
$$v:\calt|_X\rightarrow\calt,\,\, U\rightarrow \{x\in
\tx:\dd(x,U)<\dd(x,X\bs U)\}$$ (assuming that
$\dd(x,\varnothing)=\infty$ for every $x\in\tx$). Then:
\begin{enumerate}[a) ]
\item For every $U\in\calt|_X$, $v(U)\cap X=U$.
\item $v(X)=\tx$ and $v(\varnothing)=\varnothing$.
\item For every $U_1,U_2\in\calt|_X$, $U_1\subset U_2$ if and only if
$v(U_1)\subset v(U_2)$.
\item For every $U_1,\dots,U_r\in\calt|_X$, $U_1\cap\dots\cap U_r\neq\varnothing$ if and only if
$v(U_1)\cap\dots\cap v(U_r)\neq\varnothing$.
\item For every $U\in\calt|_X$, $U=\varnothing$ if and only if $v(U)=\varnothing$
\item For every $U_1,U_2\in\calt|_X$, $v(U_1\cap U_2)=v(U_1)\cap
v(U_2)$.

\end{enumerate}
\end{prop}

\begin{proof}
a)-c) are easy to check. e) is a consequence of a) and b). d) is a
consequence of e) and f). It suffices to prove f).

Since $U_1\cap U_2\subset U_1$, c) shows that $v(U_1\cap U_2)\subset
v(U_1)$. By the same reason, $v(U_1\cap U_2)\subset v(U_2)$. Hence,
$v(U_1\cap U_2)\subset v(U_1)\cap v(U_2)$.

Fix $x\in v(U_1)\cap v(U_2)$. Then, $\dd(x,U_1)<d(x,X\bs U_1)$ and
$\dd(x,U_2)<d(x,X\bs U_2)$. Choose $\epsil>0$ such that $\dd(x,X\bs
U_1)-\dd(x,U_1)>\epsil$ and $\dd(x,X\bs U_2)-\dd(x,U_2)>\epsil$.

Take $y\in X$ such that $\dd(x,y)<\dd(x,X)+\frac{\epsil}{2}$. For
every $z\in X\bs U_1$, we have $\dd(y,z)\geq
\dd(x,z)-\dd(x,y)>\dd(x,X\bs
U_1)-\left(\dd(x,X)+\frac{\epsil}{2}\right)\geq \dd(x,X\bs
U_1)-\dd(x,U_1)-\frac{\epsil}{2}>\epsil-\frac{\epsil}{2}=\frac{\epsil}{2}$.
Then, $\dd(y,X\bs U_1)> \frac{\epsil}{2}$ and hence, $y\in U_1$. By
the same reason, $y\in U_2$. Therefore:
\begin{equation}\label{eq49956a}
y\in X \textrm{ and }\dd(y,x)<\dd(x,X)+\frac{\epsil}{2}\Rightarrow
y\in U_1\cap U_2\end{equation}

From (\ref{eq49956a}) we deduce
\begin{equation}\label{eq49956b}
\dd\big(x,X\bs(U_1\cap U_2)\big)\geq\dd(x,X)+\frac{\epsil}{2}
\end{equation}

Fix $\delta>0$ and take $y'\in X$ such that
$\dd(x,y')<\dd(x,X)+\min\left\{\delta,\frac{\epsil}{2}\right\}$. By
(\ref{eq49956a}), $y'\in U_1\cap U_2$. Hence,
$$\dd(x,X)\leq\dd(x,U_1\cap U_2)<\dd(x,X)+\delta\textrm{ for every }
\delta>0$$ and we get  $\dd(x,U_1\cap U_2)=\dd(x,X)$. We conclude
from (\ref{eq49956b}) that $\dd(x,U_1\cap U_2)<\dd\big(x,X\bs
(U_1\cap U_2)\big)$, hence that $x\in v(U_1\cap U_2)$ and finally
that $v(U_1)\cap v(U_2)\subset v(U_1\cap U_2)$.\end{proof}

\begin{rems}
The function $v$ defined above is called ``Ext'' in
\cite{tretracts}, pag 125.

\item The function described in Proposition 2.7 of \cite{cc}, pag 3711, satisfies properties a)-d) of proposition \label{propiedadesv}.
\end{rems}

\begin{prop}\label{basicodefrec2} Let $\xid$ be a metrizable compactification pack, let $\calt$ be the topology of $\tx$ and consider $(\hx,\ceo)$.

Suppose that $\{W_n\}_{n=0}^\infty$ is a sequence of open
neighborhoods of $X$ such that $W_0=\tx$, $W_0\supset\overline
W_1\supset W_1\supset\overline W_2\supset W_2\supset \dots$ and
$\bigcap_{n=0}^\infty W_n=X$.

Suppose that $\{\alpha_n\}_{n=0}^\infty$ is a family of open covers
of $X$ and let, for every $n$, $\beta_n=\{v(U):U\in\alpha_n\}$,
where $v:\calt|_X\rightarrow\calt$ is a map satisfying properties
a)-d) of Proposition \ref{propiedadesv}.

Consider the cover $\alpha(\{\beta_n\},\{W_n\})$ defined in
Proposition \ref{basicodefrec}. Then,
\begin{enumerate}[a) ]
\item $\mult\alpha(\{\beta_n\},\{W_n\})\leq\sup_{n\in\Nset\cup\{0\}}\mult(\alpha_n,\alpha_{n+1})$.

\item If $\mesh\alpha_i\rightarrow 0$, then $\lim_{m,n\rightarrow
0}\mesh\{V\cap W_m:V\in\beta_n\}=0$ and, consequently,
$\alpha(\{\beta_n\},\{W_n\})$ is a uniform cover of
$(\hx,\ceo)$.\end{enumerate}
\end{prop}

\begin{proof} For short, denote $\alpha(\{\beta_n\},\{W_n\})$ by
$\alpha$. Property a) of Proposition \ref{propiedadesv} states that,
for every $n$, $v:\alpha_n\rightarrow\beta_n$ is bijection and
property d) states that, for every $n_1,\dots,n_r$:

$$\mult(\beta_{n_1},\cdots,\beta_{n_r})=\sup\left\{\sum_{k=1}^r
\#B_k:B_k\subset\beta_{n_k}\forall k, \bigcap_{k=1}^r\bigcap_{V\in
B_k}V\neq\varnothing\right\}=$$ $$\sup\left\{\sum_{k=1}^r
\#\left\{v(U):U\in A_k\right\}:A_k\subset\alpha_{n_k}\forall k,
\bigcap_{k=1}^r\bigcap_{U\in A_k}v(U)\neq\varnothing\right\}=$$
\begin{equation}\label{igualamosmult}\sup\left\{\sum_{k=1}^r
\#A_k:A_k\subset\alpha_{n_k}\forall k, \bigcap_{k=1}^r\bigcap_{U\in
A_k}U\neq\varnothing\right\}=\mult(\alpha_{n_1},\cdots,\alpha_{n_r})\end{equation}

Fix $x\in \hx$. It is easy to check that if $G\in\alpha$ with $x\in
G$, then $G=V\cap (W_k\bs \overline W_{k-2})$, where $V\in \beta_k$,
$k\in\{N,N+1\}$ and $N=\max\{n\in\Nset:x\in W_n\}$. Then,
$\mult_x\alpha=\{G\in\alpha:x\in G\}\leq\#\{V\cap (W_N\bs \overline
W_{N-2}):V\in \beta_N,x\in V\}+\#\{V\cap (W_{N+1}\bs \overline
W_{N-1}):V\in \beta_N,x\in V\}\leq\#\{V:V\in \beta_N,x\in
V\}+\#\{V:V\in \beta_N,x\in
V\}\leq\mult(\beta_N,\beta_{N+1})=\mult(\alpha_N,\alpha_{N+1})\leq\sup_{n\in\Nset\cup\{0\}}\mult(\alpha_n,\alpha_{n+1})$.
Taking supreme over $x$ we get a).

Assume $\mesh\alpha_n\rightarrow 0$. Suppose $\mesh\{V\cap
W_m:V\in\beta_n\}\not\rightarrow 0$ when $m,n\rightarrow\infty$.
Then, there exists $\epsil>0$ and two subsequences $\{n_k\}$ and
$\{m_k\}$ such that $\mesh\{V\cap W_{m_k}:V\in\beta_{n_k}\}>\epsil$.
For every $k$, take $U_k\in\alpha_k$ with $\diam v(U_k)\cap
W_{m_k}>\epsil$ and take $x_k,y_k\in v(U_k)\cap W_{m_k}$ with
$\dd(x_k,y_k)>\epsil$.

Since $\tx$ is compact, we may suppose, by taking subsequences if
necessary, that $x_k\rightarrow x$ and $y_k\rightarrow y$ for some
$x,y\in\tx$. Then,
\begin{equation}\label{eq948ujra}
\dd(x,y)=\lim \dd(x_k,y_k)\geq\epsil \end{equation}

For every $i$ and every $k\geq i$, we have $x_k,y_k\in
W_{m_k}\subset W_k\subset W_i$ and hence, $x,y\in \overline W_i$.
Thus, $x,y\in\bigcap_{i\in\Nset}\overline W_i=X$.

Let $N\in\Nset$ with $\mesh\alpha_N<\frac{\epsil}{3}$. Choose
$U_x,U_y\in\alpha_N$ such that $x\in U_x$ and $y\in U_y$. Observe
that $v(U_x)$ and $v(U_y)$ are two neighborhoods of $x$ and $y$
respectively. Since $\mesh\alpha_{n_k}\rightarrow 0$,
$x_k\rightarrow x$ and $y_k\rightarrow y$, it follows that there is
$k'$ such that $\mesh \alpha_{k'}<\frac{\epsil}{3}$, $x_{k'}\in
v(U_x)$ and $y_{k'}\in v(U_y)$. Since $v(U_{k'})\cap
v(U_x)\supset\{x\}\neq\varnothing$ and $v(U_{k'})\cap
v(U_y)\supset\{y\}\neq\varnothing$, we have that $U_{k'}\cap
U_x\neq\varnothing$ and $U_{k'}\cap U_y\neq\varnothing$. Take $x'\in
U_{k'}\cap U_x$ and $y'\in U_{k'}\cap U_y$ to get:
$$\xymatrix@R=.4mm{\dd(x,y)\leq \dd(x,x')+\dd(x',y')+\dd(y',y)\leq\\
\diam U_x+\diam U_{k'}+\diam U_y\leq\\
\mesh\alpha_N+ \mesh\alpha_{n_{k'}}+\mesh\alpha_N<
3\frac{\epsil}{3}=\epsil}$$ in contradiction with (\ref{eq948ujra}).

Then, $\lim_{m,n\rightarrow\infty}\mesh\{V\cap W_m:V\in\beta_n\}=0$
and hence, by Proposition (\ref{basicodefrec}), $\alpha$ is a
uniform cover.
\end{proof}

From \cite{moreno} we take the following result:

\begin{teo}\label{teomoreno} Let $(X,\dd)$ be a compact metric space with $\dim X\leq n<\infty$ and
suppose $\{\epsil_i\}_{i=1}^\infty\subset\Rset^+$. Then, there
exists a sequence of open and finite covers of $X$
$\{\alpha_i\}_{i=0}^\infty$ such that:
\begin{enumerate}[a) ]
\item $\alpha_0=\{X\}$ and $\mesh\alpha_i<\epsil_i$ for every
$i\in\Nset$.
\item $\mult(\alpha_i,\alpha_{i+1})\leq n+2$ for every $i\in\Nset\cup\{0\}$
\end{enumerate}
\end{teo}

\begin{proof} It is a consequence of Theorems 74, 81, 104 or 154 of
\cite{moreno}.\end{proof}

\begin{prop}\label{demgrave2} Let $\xid$ be a compactification pack, consider $(\hx,\ceo)$ and let $\gamma$ be a uniform cover of $\hx$. Then, there exists an open, locally
finite and uniform cover $\alpha$ (i. e. a canonical cover) such
that $\gamma\prec\alpha$ and $\mult\alpha\leq\dim X+2$.

Moreover, given a sequence of open subsets of $\tx$
$\{W_i\}_{i=0}^\infty$ with $W_0=\tx$, $W_0\supset\overline
W_1\supset W_1\supset \overline W_2\supset W_2 \supset\dots$ and
$\bigcap_{i=0}^\infty W_i=X$, we can construct such $\alpha$ by
letting $\alpha=\alpha(\{\beta_k\},\{W_{i_k}\})$ (as defined in
Proposition \ref{basicodefrec}), where $\{\beta_k\}_{k=0}^\infty$ is
a sequence of open and finite families of subsets  of $\hx$ such
that $\beta_0=\{\tx\}$, $X\subset\bigcup_{V\in \beta_i}V$ for every
$i$ and $\lim_{i,j\rightarrow 0}\mesh\{V\cap W_j:V\in\beta_i\}=0$
and $\{i_k\}_{k=0}^\infty$ is a subsequence with $i_0=0$.

\end{prop}

\begin{proof}
Let $\{W_i\}_{i=0}^\infty$ be as above (for example, consider
$W_0=\tx$ and, for every $i\in\Nset$,
$W_i=\BB\left(X,\frac{k}{2i}\right)$, where $k=\sup_{x\in
\tx}\dd(x,X)$).

If $\dim X=\infty$, this proposition is a consequence of
Propositions \ref{muchoscanonicos} and \ref{basicodefrec}. If $\dim
X=n<\infty$, by Theorem \ref{basicodefrec}, there exist a sequence
of open and finite covers $\{\alpha_i\}_{i=0}^\infty$ of $X$ with
$\alpha_0=\{X\}$, $\lim_{i\rightarrow\infty}\mesh\alpha_i=0$ and
$\mult(\alpha_i,\alpha_{i+1})\leq n+2$ for every
$i\in\Nset\cup\{0\}$.

\begin{figurapng}{.40}{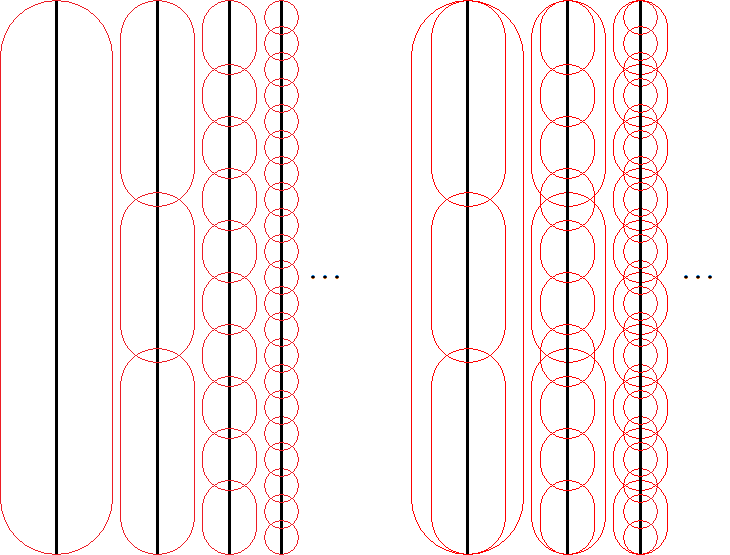}{Covers $\{\alpha_i\}_{i=0}^\infty$ and $\{\alpha_i\cup\alpha_{i+1}\}_{i=0}^\infty$ of $[0,1]$}
\centerline{with $\mult\alpha_i\leq 2$ and
$\mult(\alpha_i,\alpha_{i+1})\leq 3$ $\forall
i\in\Nset\cup\{0\}$}\label{coversteomoreno}
\end{figurapng}

Let $\calt$ be the topology of $\tx$, consider
a map $v:\calt|_X\rightarrow\calt$ satisfying properties a)-d) of
Proposition \ref{propiedadesv} and suppose
$\beta_i=\{v(U):U\in\alpha_i\}$ for every $i\in\Nset\cup\{0\}$.

By Proposition \ref{basicodefrec2},
$\lim_{i,j\rightarrow\infty}\mesh\{V\cap W_j:V\in\beta_i\}=0$ and,
by Lemma \ref{basicodefrec},  there exists a subsequence
$\{i_k\}_{k=0}^\infty$ with $i_0=0$ such that
$\gamma\cup\{\{x\}\}_{x\in
\hx}\prec\alpha(\{\beta_k\},\{W_{i_k}\})$.

\begin{figurapng}{.40}{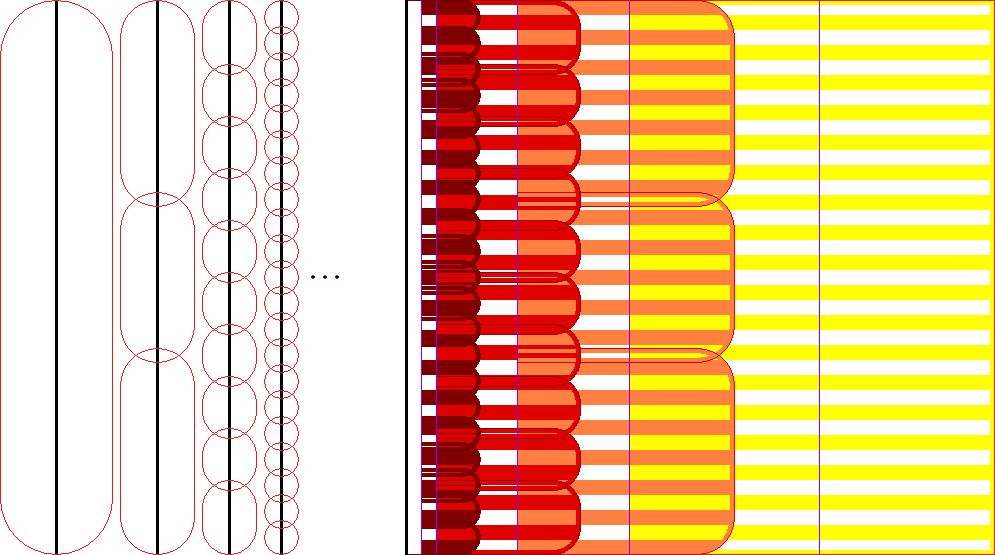}{Covers $\{\alpha_i\}_{i=0}^\infty$
of $[0,1]$ of figure \ref{coversteomoreno}} \centerline{and cover
$\alpha(\{\beta_k\},\{W_{i_k}\})$ of $[0,1]\times(0,1]$}
\end{figurapng}

$\gamma\cup\{\{x\}\}_{x\in \hx}$ is a cover of $\hx$, so it is
$\alpha(\{\beta_k\},\{W_{i_k}\})$. According to \ref{basicodefrec2},
$\alpha(\{\beta_k\},\{W_{i_k}\})$ is also an open, uniform and
locally finite cover of $\hx$ such that
$\mult\alpha(\{\beta_k\},\{W_{i_k}\})\leq\sup_{i\in\Nset\cup\{0\}}\mult(\alpha_i,\alpha_{i+1})\leq
n+2$. Particulary, it is a canonical cover of $\xid$ (see Lemma
\ref{caractcanon}).\end{proof}

\begin{proof}[Proof of (\ref{grave2})]
It is a consequence of Proposition \ref{demgrave2}.\end{proof}


\begin{thebibliography}{99}

\bibitem{Anderson} Anderson, R. D.
\emph{Topological properties of the Hilbert cube and the infinite
product of open intervals} Trans. Amer. Math. Soc. 126 (1967)
200-216

\bibitem{dr1}G. Bell, A. Dranishnikov, \emph{Asymptotic dimension in Bedlewo}.
Topology Proc. 38 (2011), 209-236


\bibitem{dr3}G. Bell, A. N. Dranishnikov, \emph{Asymptotic dimension}. Topology Appl. 155 (2008), no. 12, 1265-1296




\bibitem{zs} E. Cuchillo Ibáñez, J. Dydak, A. Kodama y M. A. Morón, \emph{C$_0$ coarse geometry of complements of Z-sets in the Hilbert cube}. Trans. Amer. Math. Soc. 360 (2008), no. 10, 5229--5246

\bibitem{cc} E. Cuchillo Ibáñez, M. A. Morón, \emph{Canonical covers and dimension of Z-sets in the Hilbert cube}. Proc. Amer. Math. Soc. 136 (2008), no. 10, 3709--3716

\bibitem{hc} T. A. Chapman, \emph{Lectures on Hilbert Cube
Manifolds}. American Mathematical society, 1976. Regional Conference
Series in Mathematics Vol. 28.

\bibitem{dr2}A. N. Dranishnikov, \emph{Asymptotic topology}. Russian Math. Surveys 55 (2000), no. 6, 1085--1129

\bibitem{dr4}A. N. Dranishnikov, J. Smith \emph{On asymptotic Assouad-Nagata dimension.}
Topology Appl. 154 (2007), no. 4, 934-952.

\bibitem{dug2} James Dugundji, \emph{An extension of Tietze's theorem}, Pacific Journal of Mathematics 1 (1952) 353-367. MR0044116

\bibitem{lss} J. Dydak and S. Hoffland, \emph{An alternative definition of coarse structures}. Topology Appl. 155 (2008), no. 9, 1013--1021

\bibitem{grv} Bernd Grave, \emph{Coarse geometry and asymptotic dimension}.
Phd Thesis.

\bibitem{grv2}
Bernd Grave \emph{Asymptotic dimension of coarse spaces}. New York
J. Math. 12 (2006)


\bibitem{hw} Hurewictz-Wallman, \emph{Dimension Theory}.
Princeton University Press, Princeton, NJ, 1941.

\bibitem{hur2} W. Hurewicz, \emph{Sur la dimension des produits
Cartésiens}, Annals of Mathematics (2), vol. 36 (1935), pp. 194-197

\bibitem{miya} Kotaro Mine, Atsushi Yamashita \emph{$C_0$ coarse structures and smirnov
compactifications} Preprint (2011) arXiv:1106.1672


\bibitem{moreno} Jesús P. Moreno Damas,
\emph{Geometría de recubrimientos: Dimensión Topológica y Estructuras Coarse $C_0$} PhD thesis. Universidad Complutense de Madrid, 2012.


\bibitem{moreno2} Jesús P. Moreno Damas,
\emph{Propiedades de la geometría $C_0$ a gran escala con
compactificación de Higson metrizable} Trabajo para la obtención del
DEA. Universidad Complutense de Madrid, 2007.

\bibitem{mor} K. Morita, \emph{On the Dimension of Product Spaces}, American Journal of Mathematics, Vol. 75, No. 2, 205-223, Apr. 1953



\bibitem{cg} John Roe, \emph{Lectures on Coarse Geometry}.
American Mathematical Society, 2003. University Lecture Series Vol.
31.

\bibitem{cg3}John Roe, \emph{Corrections to Lectures on Coarse
Geometry},\par http://www.math.psu.edu/roe/writings/correction.pdf

\bibitem{tretracts} Sze-Tsen Hu \emph{Theory of Retracts}. Wayne State University Press, 1965

\bibitem{wright} Nick Wright, \emph{C$_0$ Coarse Geometry}. PhD thesis. Penn State University 2002.

\bibitem{wright3}Wright, Nick \emph{$C_0$ coarse geometry and scalar curvature}. J.
Funct. Anal. 197 (2003), no. 2, 469-488.

\end{thebibliography}
\end{document}